\documentclass[reqno]{amsart}

\usepackage{wasysym}
\usepackage{amssymb}
\usepackage{latexsym}
\usepackage{amsmath}
\usepackage{enumerate}
\usepackage{amsthm}
\usepackage{wasysym}
\usepackage{stmaryrd}
\usepackage{mathrsfs}
\usepackage[usenames]{xcolor}
\usepackage{pifont}
\usepackage{tikz}
\usepackage{cases}

\newcommand{\qee} {\hspace*{2mm}\hfill \ding{109}}

\renewcommand{\phi}{\varphi}

\definecolor{uuxgreen}{cmyk}{1,0,0.75,0}
\definecolor{uuxred}{cmyk}{0.2,1,0.9,0.1}
\definecolor{uuyblue}  {cmyk}{0.9,0.55,0,0}

\newcommand{\qedright}{\belowdisplayskip=-12pt}

\newtheorem{theorem}{Theorem}[section]
\newtheorem{define}[theorem]{Definition}
\newenvironment{definition}{\begin{define} \rm}{\qee\end{define}}
\newtheorem{exa}[theorem]{Example}
\newenvironment{example}{\begin{exa} \rm}{\qee\end{exa}}
\newtheorem{exerc}[theorem]{Exercise}

\newtheorem{conj}[theorem]{Conjecture}

\newtheorem{ques}[theorem]{Open Question}
\newenvironment{question}{\begin{ques} \rm}{\qee\end{ques}}
\newtheorem{lem}[theorem]{Lemma}
\newenvironment{lemma}{\begin{lem} \it}{\end{lem}}
\newtheorem{cor}[theorem]{Corollary}

\newtheorem{rem}[theorem]{Remark}
\newenvironment{remark}{\begin{rem} \rm}{\qee\end{rem}}

\newcommand{\medent}{\medskip\noindent}
 \newcommand{\tupel}[1]{{\langle #1 \rangle}}
\newcommand{\verz}[1]{\{ #1 \}}

\newcommand{\gnum}[1]{{\ulcorner #1 \urcorner}}

\newcommand{\concat}{%
	\mathord{
		\mathchoice
		{\raisebox{1ex}{\scalebox{.7}{$\frown$}}}
		{\raisebox{1ex}{\scalebox{.7}{$\frown$}}}
		{\raisebox{.7ex}{\scalebox{.5}{$\frown$}}}
		{\raisebox{.7ex}{\scalebox{.5}{$\frown$}}}
	}
}

\DeclareMathSymbol{\mlq}{\mathord}{operators}{``}
\DeclareMathSymbol{\mrq}{\mathord}{operators}{`'}

\usepackage[usenames]{xcolor}
\usepackage{proof}
\usepackage{mathrsfs}
\usepackage[babel]{csquotes} %B: for \enquote{}
\usepackage[T1]{fontenc} %B: for \guillemotleft
\usepackage{yhmath} %B: for \wideparen
\usepackage{multicol} %B: for multicolumns
\usepackage{mdframed} %B: for siderules

\newcommand{\nub}{{\mathfrak n}}

\newcommand{\zero}{{\sf 0}}

\usepackage{natbib}

\newcommand{\noshow}[1]{{}}
\newcommand{\nusu}[1]{{\lceil #1 \rceil}}
\newcommand{\nusus}[1]{{\lfloor #1 \rfloor}}
\newcommand\Tstrut{\rule{0pt}{2.6ex}}         % = `top' strut
\newmdenv[
topline=false,
bottomline=false,
skipabove=\topsep,
skipbelow=\topsep
]{siderules}
\title[Self-reference upfront]{Self-reference upfront\\ 
{\footnotesize A Study of Self-referential G\"odel Numberings}}



\author{Balthasar Grabmayr}
\address{Department of Philosophy,
	Humboldt University of Berlin,
	Unter den Linden 6,	
	10099 Berlin,
	Germany}
\email{balthasar.grabmayr@gmx.net}
\author{Albert Visser}
\address{Philosophy, Faculty of Humanities,
	Utrecht University,
	Janskerkhof 13,
	3512BL~~Utrecht, The Netherlands}
\email{a.visser@uu.nl}
\date{\today}

\begin{document}
	
	\begin{abstract}
		In this paper we examine various requirements on the formalisation choices under which self-reference can be adequately formalised in arithmetic. In particular, we study self-referential numberings, which immediately provide a strong notion of self-reference even for expressively weak languages. The results of this paper suggest that the question whether truly self-referential reasoning can be formalised in arithmetic is more sensitive to the underlying coding apparatus than usually believed. 
		As a case study, we show how this sensitivity affects the formal study of certain principles of self-referential truth.
	\end{abstract}
	
	\maketitle
	
	\section{Introduction}

	Notions of self-reference feature prominently in the philosophical literature, yet they are notoriously elusive and mostly left imprecise. The study of certain logical and philosophical 
	issues, however, such as semantic paradoxes and the semantics of both natural and formal languages,
	requires a precise understanding of self-reference. For instance, an adequate evaluation 
	of the contentious hypothesis that circularity lies at the root of all semantic paradoxes requires a 
	satisfactory and precise explication of self-reference (see \citep{Leitgeb2002}). Following the 
	wide-spread custom of \textit{arithmetisation}, the explication and investigation of self-reference 
	is therefore typically carried out in a formal arithmetical framework.
	
	There are essentially two ways an arithmetical sentence can be taken to (directly) refer to itself, namely, in virtue of containing a term denoting (the code of) itself, or by means of quantification \citep{Leitgeb2002, Picollo2018}.\footnote{\label{ftn:Kripkeseminar}In his Princeton seminar on truth, September 29, 1982, Saul Kripke distinguishes these two notions of self-reference from \emph{demonstrative} self-reference, which is obtained by indexicals, such as the word \enquote{this}, and therefore is unattainable in standard formalisms of arithmetic. For proposals how such indexicals can be added to systems of arithmetic see \citep{Smullyan1984} and \citep{Fraassen1970}. We are grateful to Allen Hazen for sharing with us his transcript of Kripke's seminar and to Saul Kripke for his permission to use them.} In this paper we are only concerned with the first kind of self-reference, also called \enquote{self-reference by mention} or in short \enquote{m-self-reference} \cite[p.~581]{Picollo2018}. \cite{HalbachVisser1} trace this notion of self-reference back to the famous exchange of Kreisel and Henkin regarding Henkin's question as to whether sentences that state their own $\mathsf{PA}$-provability are provable in $\mathsf{PA}$. \cite{Heck2007}, for instance, considers m-self-reference the only legitimate way to formalise truly self-referential reasoning in arithmetic. A precise notion of m-self-referentiality in an arithmetical framework rests on three formalisation choices:
	\begin{enumerate}[i.]
		\item a formal language
		\item a G\"odel numbering
		\item a naming device
	\end{enumerate}
	
	Given the central importance of m-self-reference in the literature, the primary concern of this paper is to examine the formalisation choices under which m-self-reference is attainable. That is, we ask for which choices of (i)-(iii) we can find for each formula $A(x)$, with free variable $x$, a closed term $t$ which denotes the G\"odel code of $A(t)$.
	
	Let ${\mathcal{L}^0}$ be the arithmetical language which has $\zero$, $\mathsf{S}$, $+$ and $\times$ as its non-logical vocabulary. Let $\xi$ be a standard numbering of ${\mathcal{L}^0}$ and consider the canonical naming device which maps each number $n$ to its standard numeral $\underline{n}$. Taken together, these canonical formalisation choices do not, generally, permit the construction of m-self-referential sentences.\footnote{See Lemma~\ref{lem:unattainabilityofSR} for a generalisation of this fact.} There are two well-known routes to overcome this.
	
	The first route consists in enriching the arithmetical language by function symbols for primitive recursive functions, such that the resulting language ${\mathcal{L}}^+$ contains a term $\delta(x)$ which represents a diagonal function~$D$. Given a standard numbering $\xi^+$ of ${\mathcal{L}}^+$, a diagonal function $D$, adequate for these choices, maps (the $\xi^+$-code of) an ${\mathcal{L}}^+$-formula $A(x)$ with $x$ free to (the $\xi^+$-code of) its diagonalisation $A(\delta(\underline{\xi^+(A)}))$. This language extension can be achieved in different ways. Since $D$ is primitive recursive, one may simply add a function symbol for $D$ to ${\mathcal{L}^0}$, or one can add function symbols for other p.r.\ functions such as the substitution function and the numeral function and then represent $D$ as a complex term, etc. Let $\mathsf{R}_{{\mathcal{L}}^+}$ be the result of adding all true equations of the form $t = \underline{n}$ to the Tarski-Mostowski-Robinson theory $\mathsf{R}$, where $t$ is a closed ${\mathcal{L}}^+$-term. The Strong Diagonal Lemma then provides the existence of m-self-referential sentences:
	
	\begin{lem}[Strong Diagonal Lemma for ${\mathcal{L}}^+$]
		\label{classicdiaglemma}
		For every ${\mathcal{L}}^+$-formula $A(x)$ with $x$ as a free variable, there exists a closed ${\mathcal{L}}^+$-term $t$ such that $\mathsf{R}_{{\mathcal{L}}^+} \vdash t = \underline{\xi(A(t))}$.
	\end{lem}
	
	The second route is based on so-called \textit{self-referential} G\"odel numberings. A numbering $\xi$ is called self-referential if, for any formula $A(x)$, we can find a number $n$ such that $n = \xi(A(\underline n))$. Self-referential numberings thus immediately provide m-self-referential sentences, without the need of extending the language (and the theory). The idea of self-referential G\"odel numberings can be traced back to \citep[footnote~6]{Kripke1975} and \citep[p.~80]{Feferman1984}. Constructions of self-referential numberings were given in Kripke's 1982 Princeton seminar (see footnote~\ref{ftn:Kripkeseminar}), in \citep{viss:sema89, viss:sema04}, in \citep{Heck2007} and in the Appendix of \citep{HalbachVisser2}. The last construction is for efficient numerals: see below.
	
	The second route is typically taken to be contrived and unsatisfactory since the self-referential numberings existing in the literature to date do not satisfy certain desirable properties. One such property is \emph{monotonicity}, which requires the G\"odel number of an expression to be larger than the G\"odel numbers of its sub-expressions. \cite{Halbach2018}, for instance, does not consider non-monotonic G\"odel numberings as adequate formalisation choices.\footnote{We do think more argument is needed to philosophically justify constraints on G\"odel numberings. However, our present interest is not in that discussion, but in the further implications of such constraints.} Indeed, in the study of self-reference, monotonicity is prevalently required to hold for any \textit{reasonable} or \textit{adequate} numbering. It is, thus, widely believed that for adequate choices of G\"odel numberings m-self-reference is not attainable in arithmetic formulated in ${\mathcal{L}^0}$ (see \citep{Heck2007}). Accordingly, the natural setting to formalise self-referential reasoning is typically based on the first approach, by increasing the expressiveness of the language ${\mathcal{L}^0}$. 
	
	The aim of the present paper is to further investigate this trade-off between the (term-)expressiveness of the language and the naturalness of the underlying coding. 
	In particular, we will show that the received view regarding the unattainability of m-self-reference in ${\mathcal{L}^0}$ is committed to much stronger assumptions on the coding apparatus than usually assumed in the literature.
	
	We will start by investigating which requirements on G\"odel numberings are sufficient to rule out self-referential numberings as admissible formalisation choices. We will first observe that when employing standard numerals, self-referential G\"odel numberings indeed cannot be monotonic (see Section~\ref{subsection:monotonicity}). Thus, by Halbach's standards, self-referential G\"odel numberings would be disqualified as adequate formalisation choices. They could be, at most, a technical tool providing an alternative proof of the G\"odel Fixed Point Lemma and, thus, of the Incompleteness Theorems. (See for instance \citep[Lemma 4.12]{Grabmayr2020}.)
	
	This is not the end of the story, however. While the attainability of m-self-referential sentences is not affected by the underlying naming device (as will become clear from inspection of Definition~\ref{def:attainabilityofSR}), we will show that the (in-)compatibility of monotonicity and self-referentiality of numberings depends crucially on this aspect of formalisation. In fact, we will show in this paper that we can produce effective self-referential monotonic numberings by basing the definition of \emph{self-referential} on efficient instead of standard numerals (see Section~\ref{section:construction2}~\&~Appendix~\ref{section:construction1}). That is, for any formula $A(x)$, we can find a number $n$ such that $n = \xi(A(\overline n))$, where $\overline{n}$ denotes the efficient numeral of $n$ (see Section~\ref{subsection:numerals}). Setting $t := \overline{n}$ hence yields $\mathsf{R} \vdash t = \underline{\xi(A(t))}$. The adequacy constraints on numberings have thus to be strictly more restrictive than effectiveness and monotonicity in order to rule out self-referential numberings and, in particular, m-self-referentiality in ${\mathcal{L}^0}$.
	
	In Section~\ref{section:numeralsasquotations}, we introduce a strengthened notion of monotonicity put forward by \cite{Halbach2018}, which captures the idea that (efficient) numerals are arithmetical proxies for quotations. A monotonic coding is called strongly monotonic if the code of the G\"odel numeral of an expression is larger than the code of the expression itself. This constraint on numberings is sufficiently restrictive to exclude self-referential numberings for any numeral system. However, we will show that even strong monotonicity is not restrictive enough to exclude m-self-referentiality. In fact, we will construct an effective and strongly monotonic numbering which gives rise to the Strong Diagonal Lemma for $\mathcal{L}^0$, thus providing m-self-referential sentences formulated in $\mathcal{L}^0$.
	
	In Section~\ref{section:compadequacy}, we introduce computational constraints which are more restrictive than effectiveness and which may serve as additional adequacy constraints for numberings. A G\"odel numbering is called $\mathcal E$-adequate if it represents a large portion of syntactic relations and operations by \textit{elementary} relations and operations on $\omega$. We show that the numberings constructed in this paper are $\mathcal E$-adequate in this sense. Hence, strong monotonicity and $\mathcal E$-adequacy taken together, are once again not restrictive enough to exclude m-self-referentiality in $\mathcal{L}^0$.
	
	In Section~\ref{section:regularity}, we briefly discuss the constraint of \textit{regularity} on numberings for languages $\mathcal{L} \supseteq {\mathcal{L}^0}$ due to \cite{Heck2007}. Indeed, this constraint is sufficiently restrictive to rule out the existence of m-self-referential sentences formulated in $\mathcal{L}$. However, we construct a decent numbering which is \textit{not} regular for $\mathcal{L}^0$. Regularity thus hardly serves as a necessary constraint for admissible numberings and m-self-referentiality in $\mathcal{L}^0$ is in the clear.
	
	The obtained results suggest the following disjunctive conclusion: when formalising m-self-reference in arithmetic, the adequacy constraints on reasonable numberings are more restrictive than widely assumed, or m-self-reference can be adequately formalised in a less expressive language than usually believed.\footnote{In this paper, we do not take a stand on which of the disjuncts obtains.} More specifically, we conclude that either the constraints on reasonable numberings are more restrictive than $\mathcal E$-adequacy and strong monotonicity taken together, or m-self-reference is already attainable in ${\mathcal{L}^0}$.
	
	We close, in Section~\ref{section:discussion}, by showing how these results bear on the study of axiomatic truth theories. In particular, we show that the constraints of $\mathcal E$-adequacy and (strong) monotonicity taken together are not sufficient to determine the consistency of certain type-free truth theories. Thus, the formalisation of certain informal principles of truth in an arithmetical setting is highly sensitive to the underlying formalisation choices. These results raise doubts as to what extent such axiomatic theories can be taken to faithfully reflect informal reasoning regarding the underlying principles of truth.
	
	Finally, we hope to provide entertaining examples of G\"odel numberings which bring to light some surprising subtleties regarding the interaction between self-reference and the employed formalisation devices.
	
	\section{Technical Preliminaries}
	
	In this section, we introduce the necessary basic notions concerning syntax, theories and numberings.
	
	\subsection{Languages} 
	\label{subsection:language}
	
	Languages can be represented in many ways: as free algebras, as sets of strings, as labeled directed acyclic graphs, as \dots\ These choices are both philosophically and technically important. In our context of the study of G\"odel numberings, the choice of a format for the language will often suggest a particular choice of numerical representation.
	E.g., labeled directed acyclic graphs can be modeled in the hereditarily finite sets and we
	can map these sets into numbers using the Ackermann coding. It would be very natural to base a G\"odel numbering on this idea. In this paper, we will mainly employ the algebraic perspective and on the string perspective. This does not reflect a philosophical standpoint, but is just a limitation dictated by length. In Section~\ref{subsection:nondominatingnumberings}, we will briefly consider the idea of sharing material in syntax.
	
	Let $\mathcal{L}^0$ be the language of first-order arithmetic, which contains $\bot$, $\top$, ${=}$, $\neg$, $\wedge$, $\vee$, ${\rightarrow}$, $\forall$ and $\exists$ as logical constants, as well as the non-logical symbols $\zero$, $\mathsf{S}$, $+$ and $\times$. 
	The infix expressions of ${\mathcal{L}^0}$ are given as follows. 
	\begin{itemize}
		\item
		$x::= {\sf v} \mid x'$
		\item
		$t ::= x \mid \zero \mid {\sf S}t \mid (t+t) \mid (t\times t)$
		\item
		$A ::= \bot \mid \top \mid t=t \mid \neg A \mid (A \wedge A) \mid (A \vee A) \mid (A \to A) \mid \forall x\,A \mid \exists x\, A$
	\end{itemize}
	
	\noindent
	Alternatively, we will sometimes consider ${\mathcal{L}^0}$ to be given in Polish notation:
	\begin{itemize}
		\item
		$x::= {\sf v} \mid x'$
		\item
		$t ::= x \mid \zero \mid {\sf S}t \mid {\sf A} t t \mid {\sf M} t t$
		\item
		$A ::= \bot \mid \top \mid {=}tt \mid \neg A \mid \wedge A A \mid \vee A A \mid {\to} A A \mid \forall x\,A \mid \exists x\, A$
	\end{itemize}
	
	Let $\mathbb{N}$ be the standard interpretation of $\mathcal{L}^0$, with $\omega$ as its domain. 
	In this paper, we consider arbitrary languages $\mathcal{L} \supseteq {\mathcal{L}^0}$ with finite signature such that each constant symbol $c$ and function symbol $f$ of $\mathcal{L}$ has an intended interpretation $c^{\mathbb{N}}$ and $f^{\mathbb{N}}$ in $\mathbb{N}$. Thus, in particular, the evaluation function $\mathsf{ev}$ is well-defined on closed terms of $\mathcal{L}$. For any such language $\mathcal{L}$, both definitions above extend in the obvious way. When we do not specify the employed notation system, we assume $\mathcal{L}$ to be given in infix notation.
	
	\subsection{Theories}
	We will mainly consider theories in the language of arithmetic. Our basic theory is
	be the Tarski-Mostowski-Robinson theory $\mathsf{R}$ (see \cite[p.~53]{TarskiMostowskiRobinson}).
	We will extend {\sf R} in a standard way to $\mathsf{R}_{{\mathcal{L}}}$ by adding all true equations of the form $t = \underline{n}$ to $\mathsf{R}$, where $t$ is a closed ${\mathcal{L}}$-term. Of course, this requires the intended interpretations of the new function symbols in the background.
	
	\subsection{G\"odel Numberings}
	
	Let $S$ be a domain which permits a robust notion of effectiveness for functions $\xi \colon S \to \omega$.\footnote{Traditionally, the notion of effectiveness for functions with domain $S$ is reduced to Turing computability or recursiveness by coding the elements of $S$ as strings or numbers respectively. However, different coding devices in general yield different extensions of effective functions. Here, we require that $S$ is a domain for which the notion of effectiveness does not depend on such choices. For instance, $S$ may be taken to be a space of finite objects, in the sense of \cite{Shoenfield1972}. In this paper, we will be solely concerned with domains of strings and other syntactic expressions which permit robust notions of effectiveness. Due to limitations of space we leave a discussion of this important matter for another occasion.} We say that a function $\xi \colon S \to \omega$ is a \textit{G\"odel numbering} or \textit{coding} of $S$, if $\xi$ is 
	injective and effective.
	We also call $\xi(A)$ the ($\xi$-)\textit{code} of $A$ (for terms and formul{\ae} $A$). In this paper we consider numberings of languages $\mathcal{L}$ given in 
	infix or Polish notation. We note that choice of infix language versus Polish language is, in a sense, immaterial, since the 
	two languages / representations of the language are connected by a standard bijection. 
	
    We occasionally also consider a language as embedded in a set of strings. Then, a numbering of the strings over the given finite alphabet will induce a numbering of the language. For instance, the infix expressions of ${\mathcal{L}^0}$ can be conceived of as strings over an alphabet containing $17$ symbols, while its Polish notations can be formulated as strings over an alphabet with $15$ symbols. (So, when we consider the language as embedded in strings the difference between infix and Polish suddenly does have some role.)
	
	We note that, e.g., if we would think of formul{\ae} as labeled directed acyclic graphs (dags), the numbering's domain $S$ could be the totality of all finite labeled dags with labels in a fixed alphabet. Such choices often reflect a syntax theory. If we view syntax as \emph{sui generis}, $S$ will usually be the set of expressions itself. If we 
	view the syntactic objects as specima of a wider variety $X$, we will usually take $S :=X$.
	
	\begin{remark}	
		We could allow one extra degree of freedom for G\"odel numberings: 
		we could drop functionality. 
		For example, suppose we want to an existing
		coding of syntax in the finite sets to define our G\"odel numbering
		via coding the finite sets as numbers. There are various ways to
		code the finite sets. Suppose we do it by interpreting sequences first and then
		ignoring the order of the components. This gives us a non-functional G\"odel numbering.
		Note that it would still be injective.
		For the purposes of the present paper we will not need the extra flexibility of non-functionality.
	\end{remark}

	\begin{remark}
		The G\"odel numbering employed in \citep{fefe:arit60} is a nice example of a G\"odel numbering that looks directly at the language-as-algebra without considering it as embedded in strings.
		For Feferman's coding, the choice between infix and Polish is irrelevant.
	\end{remark}
	
	An important example of a numbering of strings is the \textit{length-first ordering}. Let $\mathcal A^\ast$ be a finite alphabet and suppose some ordering of $\mathcal A$ is given. We  order the strings of $\mathcal A$  using the \emph{length-first ordering} $(\alpha_n)_{n\in \omega}$ in which	we enumerate the strings according to increasing length, where the strings of same length are ordered alphabetically. We set $\mathfrak g(\alpha) = n$ if $\alpha = \alpha_n$.\footnote{We note that $\mathfrak{g}$ has the alphabet and the ordering
		on the alphabet as hidden parameters. To make our notation not too heavy we will always suppress these data.} We write $|\alpha|$ for the length of $\alpha$. 
	
	Throughout this paper, we will use the following basic fact about the length-first ordering.
	
	\begin{lemma}
		\label{lem:lengthfirstestimates}
		Suppose the alphabet $\mathcal A$ has $N \geq 2$ letters. We then have:
		\[ \frac{N^{|\alpha|}-1}{N-1} \leq \mathfrak g(\alpha)  <  \frac{N^{|\alpha|+1}-1}{N-1}.\]
		It follows that $|\alpha| \leq \mathfrak g(\alpha)$.
		Moreover, whenever $|\alpha| < |\beta|$, then
		$\mathfrak g(\alpha) < \mathfrak g(\beta)$.
	\end{lemma} 
	
	\begin{proof}
		Clearly, for any string $\alpha$ of $\mathcal A$ we have
		\[
		\overbrace{1\cdots 1}^{|\alpha| \times} \leq 
		\mathfrak g(\alpha) \leq \overbrace{1\cdots 1}^{|\alpha|+1\; \times},
		\]
		where, for any $m$, $\overbrace{1\cdots 1}^{m \times}$ is considered as an $N$-adic notation. The claim follows immediately from a well-known property of geometric series:
		\qedright
		\[
		\overbrace{1\cdots 1}^{m \times} = \sum_{i=0}^{m - 1} N^i = \frac{N^m-1}{N-1}
		\]
	\end{proof}

	\subsection{Naming Devices}
	\label{subsection:numerals}
	
	Let $\mathsf{ClTerm}_{\mathcal{L}}$ denote the set of closed ${\mathcal{L}}$-terms. We call a function ${\nu \colon \omega \to \mathsf{ClTerm}_{\mathcal{L}}}$ a \textit{numeral function for} $\mathcal{L}$, if $\nu$ is injective and effective and the closed term $\nu(n)$ has value $n$, for each $n \in \omega$. We call $\nu(n)$ the \emph{$\nu$-numeral} of $n$. Standard numerals form a canonical choice of a numeral function and are defined as follows.
	\begin{itemize}
		\item
		$\alpha ::= \zero \mid {\sf S}\alpha$.
	\end{itemize}
	
	\noindent
	We write $\underline n$ for the ordinary, or standard, numeral with value $n$, and also write $\underline{\cdot}$ for the standard numeral function. These numerals provide unique normal forms for the natural numbers.
	These normal forms reflect the fact that we naturally read the axioms for addition and multiplication as directed.
	We have $(x+\zero) \rightsquigarrow x$, $(x+{\sf S}y) \rightsquigarrow {\sf S}(x+y)$, $(x\times \zero) = \zero$, $(x \times {\sf S}y) \rightsquigarrow ((x\times y)+x)$.
	The numerals are the normal forms for the closed terms in this reduction system. 
	
	In the literature on weak systems, however, another kind of numeral
	is used that corresponds to binary or dyadic notations. The main reason for this is simply that
	the arithmetisation of the numeral function for standard numerals and the usual G\"odel numberings
	exhibits exponential growth. This holds for example for the G\"odel numbering
	$\mathfrak g$ based on the length-first ordering.
	The numeral function for the numerals in the other style is more efficient.
	Thus, these alternative numerals are also called \emph{efficient numerals}. 
	We define efficient numerals for dyadic representations as follows.
	\begin{itemize}
		\item
		$\alpha ::= \zero \mid {\sf S}({\sf SS}\zero \times \alpha) \mid  {\sf SS}({\sf SS}\zero \times \alpha)$.
	\end{itemize}
	
	\noindent
	We write $\overline n$ for the efficient numeral with value $n$ and write $\overline{\cdot}$ for the efficient numeral function. We note that efficient numerals also correspond more naturally to reduction systems in a different
	signature, e.g.\ a theory of strings formulated
	with two successors. 
	
	\begin{remark}
		Consider the G\"odel numbering $\mathfrak g$ for the language
		$\mathcal{L}^0$ based on the length-first ordering, where we assume we used infix notations. We note that the cardinality of the alphabet is 17. Consider any number $n$. 
		
		Clearly, we have $|\underline n|= n+1$. So, $\frac{17^{n+1}-1}{16} \leq \mathfrak g(\underline n)$. So $\mathfrak g(\underline n)$
		has an exponential lower bound in $n$.
		
		We give an upper bound for $\mathfrak g(\overline n)$. Let the length of the dyadic notation of $n$ be $k$. We have:
		$ 2^k-1 \leq n < 2^{k+1}-1$.
		We can estimate the length of the dyadic numeral $\overline n$
		by considering the worst case where the dyadic notation of
		$n$ is a string of 2's. We find that $|\overline n| \leq 8k+1$.
		Thus, using that $17<2^5$, we see that:
		\[\mathfrak g(\overline n) < \frac{17^{8k+2} -1}{16} \leq 64\cdot 2^{40k} \leq 64\cdot (n+1)^{40}.\]
		So, we see that $ \mathfrak g(\overline n)$  has a polynomial upper bound in $n$.
	\end{remark}
	
	\section{Formalisations of m-Self-Reference}
	
	The notion of self-reference considered here is defined by means of a reference relation on sentences. According to \cite{Leitgeb2002}, \enquote{a singular sentence might [...] be defined to refer to all the referents of all of its singular terms, and only to them} (p.~4). A sentence is then said to be self-referential if and only if it refers to itself. By employing a G\"odel numbering, this notion of self-reference can be formalised in an arithmetical framework according to the following definition.\footnote{Picollo distinguishes the formal notions of m-reference \textit{simpliciter} (\citeyear{Picollo2018}) and alethic m-reference (\citeyear{Picollo2019a}). While the former is intended to capture a pre-theoretical notion of reference \textit{simpliciter} as indicated by the above quotation by Leitgeb, the latter is more restrictive and specifically tailored to studying the reference patterns which underlie paradoxical expressions in the context of truth. Even though Definitions~\ref{def:mselfref} \& \ref{def:attainabilityofSR} capture m-self-reference \textit{simpliciter}, the results of this paper also apply to alethic m-self-reference.}
	
	\begin{definition}
		\label{def:mselfref}
		A sentence of the form $A(t)$ is called \emph{m-self-referential} with respect to the formula $A(x)$, if the closed term $t$ denotes (the code of) $A(t)$.
	\end{definition}
	
	This notion of self-reference is for example captured by the Kreisel-Henkin Criterion for self-reference (see \cite[p.~684]{HalbachVisser1}). According to this criterion, given a formula $A(x)$ expressing a property $P$, a sentence $A(t)$ says of itself that it has property~$P$ iff $A(t)$ is m-self-referential with respect to $A(x)$ (in the sense of the above definition).
	
	In order to make this notion of m-self-reference mathematically precise, Definition~\ref{def:mselfref} has to be specified with respect to the formalisation choices (i)-(iii). Accordingly, we call a triple $\langle \mathcal{L}, \xi, \nu \rangle$ a \textit{formalisation choice}, if
	\begin{itemize}
		\item $\mathcal{L}$ is a language as specified in Section~\ref{subsection:language};
		\item $\xi$ is a G\"odel numbering of $\mathcal{L}$;
		\item $\nu$ is a numeral function for $\mathcal{L}$.
	\end{itemize}
	
	\begin{definition}
		\label{def:attainabilityofSR}
		We say that m-self-reference is \emph{attainable} for a formalisation choice $\langle \mathcal{L}, \xi, \nu \rangle$, if for each $\mathcal{L}$-formula $A(x)$ there exists a closed $\mathcal{L}$-term $t$ such that ${\mathsf{R}_{\mathcal{L}} \vdash t = \nu(\xi(A(t)))}$.\footnote{We intend, of course, the call-by-value reading here: $\nu(\xi(A(t)))$ yields some term $u$, such that $\mathsf{R}_{\mathcal L}\vdash t=u$.} 
	\end{definition}
	
	It is easy to see that the attainability of m-self-reference is invariant regarding the choice of the numeral function. That is, m-self-reference is attainable for $\langle \mathcal{L}, \xi, \nu_1 \rangle$ iff m-self-reference is attainable for $\langle \mathcal{L}, \xi, \nu_2 \rangle$, for any language $\mathcal{L}$, numbering $\xi$ and numeral functions $\nu_1$ and $\nu_2$ for $\mathcal{L}$. Hence, instead of requiring the provability of $t = \nu(\xi(A(t)))$ in Definition~\ref{def:attainabilityofSR}, one may equivalently require that ${\mathsf{R}_{\mathcal{L}} \vdash t = \underline{\xi(A(t))}}$. The reason for considering arbitrary numeral functions as parameters in formalisation choices is that certain adequacy conditions on numberings will be shown to be sensitive to this aspect of formalisation.
	
	Moreover, instead of using provability in $\mathsf{R}_{\mathcal{L}}$, one can equivalently require in Definition~\ref{def:attainabilityofSR} that $\mathbb{N} \models t = \underline{\xi(A(t))}$, since $\mathsf{R}_{\mathcal{L}}$ proves every true equation of closed $\mathcal{L}$-terms.
	
	\subsection{Adequacy of Formalisation}
	\label{subsection:adequacy}
	
	In formal studies of truth, self-reference and the semantic paradoxes, it is common practice to employ natural numbers as theoretical proxies for syntactic expressions. Accordingly, certain designated arithmetical domains or theories serve as domains of expressions or syntax theories respectively. This practice of \textit{arithmetisation} rests on G\"odel numberings, which can be seen as translation devices between expressions and numbers. However, as a consequence of intensionality phenomena in metamathematics, not every numbering, i.e., injective and effective function, can be considered an equally adequate candidate.
	
	When formalising notions such as truth or (self-)reference over an arithmetical domain or theory, it is therefore essential that the underlying G\"odel numbering constitutes an \textit{adequate} translation device between expressions and numbers. Only then are we justified in considering the given arithmetical framework a domain or theory of expressions and in taking the formalised notions to apply to linguistic expressions as intended. In order to guarantee a faithful formalisation of the given philosophical notions, \textit{adequate choices of formalisation} are thus required.
	
	It is notoriously difficult to precisely characterise the notion of adequacy for numberings for a given metamathematical or philosophical purpose. We will therefore adopt a more modest approach in this paper by discussing certain \textit{necessary} conditions for adequate numberings which can be found in the literature on self-reference and axiomatic truth theories.
	
	\subsubsection{Monotonicity}
	
	Perhaps, the most prominent adequacy constraint is \emph{monotonicity}. Let $S$ be a set and let $\vartriangleleft$ be a strict partial order relation on $S$. A G\"odel numbering $\xi$ of $S$ is called \emph{monotonic with respect to $\vartriangleleft$}, if $\alpha \vartriangleleft \beta$ implies $\xi(\alpha) < \xi(\beta)$, for all $\alpha, \beta \in S$. Let $\mathcal{L}$ be a language given in infix or Polish notion (see Section~\ref{subsection:language}). We call a numbering $\xi$ of $\mathcal{L}$ \textit{monotonic}, if it is monotonic with respect to the strict sub-expression relation $\prec$ given on infix or Polish notations respectively. Monotonicity in this context can therefore be equivalently characterised by requiring a numbering $\xi$ of $\mathcal{L}$ to satisfy the following three conditions:
	\begin{enumerate}[M1.]
		\item \label{M1} for all $\mathcal{L}$-terms $s,t$, if $s \prec t$ then $\xi(s) < \xi(t)$;
		\item \label{M2} for all $\mathcal{L}$-formul{\ae} $A, B$, if $A \prec B$ then $\xi(A) < \xi(B)$;
		\item \label{M3} for all $\mathcal{L}$-terms $s$ and formul{\ae} $A$, if $s \prec A$ then $\xi(s) < \xi(A)$.
	\end{enumerate}
	Thus, when we consider numberings of well-formed expressions, we do not demand monotonicity with respect to sub-\emph{strings}. E.g., in case $\xi((A \wedge B))$ would be $\tupel{7,\tupel{\xi(A),\xi(B)}}$, where $\tupel{\cdot,\cdot}$ is the Cantor Pairing, the brackets do not appear in the code, so we do not need a G\"odel number for a bracket. However, when considering numberings of strings, monotonicity may be also required with respect to the (strict) sub-string relation. Clearly, any numbering which is monotonic in this sense also satisfies the conditions M1 - M3.
	
	\cite{Halbach2018} for instance does not consider non-monotonic G\"odel numerings as adequate formalisation choices. Indeed, the study of self-reference is typically based on numberings which are required to be monotonic. This requirement is explicit for instance in \cite[footnote 4]{Milne2007}, \cite[p.~33]{Halbach2014}, \cite[pp.~573f.]{Picollo2018}, \cite[footnote 4]{Picollo2019a} and \cite[footnote 3]{Picollo2019b} (where Milne requires monotonicity with respect to the sub-string relation). A good example of an application of (a strengthened version of) monotonicity is the proof of the falsity of the $\Sigma^0_1$-truth teller for fixed-point operators with certain good properties. See \cite[Theorem 7.7]{HalbachVisser2}.
	
	One way to motivate this constraint as a necessary condition of adequacy for numberings proceeds as follows. Let $\leq$ be the usual less-than (or equal) relation on numbers and let $\preceq$ denote the sub-expression relation on the given domain $\mathcal{L}$ of expressions. One may conceive of $\preceq$ as a parthood relation on $\mathcal{L}$.
	%, for instance via the natural interpretation of strings as geometrical objects. 
	One may endow $\leq$ with such a mereological interpretation 
	%proceeds similarly,
	by representing numbers as strings of strokes. 
	Another approach is to represent numbers as Von Neumann ordinals, and to understand the parthood relation on sets as the containment relation $\subseteq$ (see e.g.~\citep{Lewis1991}). Then $\leq$ is indeed a parthood relation on numbers.\footnote{Note that this interpretation breaks down for other reasonable representations, such as Zermelo ordinals or equivalence classes of finite sets under the equivalence relation of equinumerosity. For a critical discussion of monotonicity, see \citep[Section 3.4]{Grabmayr2020}.}	Taking parthood as an important structural feature of the domain of expressions, adequate numberings may then be required to preserve structure in virtue of representing the parthood relation $\preceq$ on $\mathcal{L}$ by a relation on $\omega$ contained in the parthood relation $\leq$ on numbers, which is tantamount to requiring monotonicity.	We will introduce further adequacy conditions for numberings in later sections of this paper.\\
	
	\subsubsection{Numerals}
	
	Thus far, we have only discussed adequacy with respect to G\"odel numberings. It remains to be clarified which numeral functions constitute adequate formalisation choices. One approach is to characterise a numeral function as adequate iff it corresponds to an \emph{acceptable} or \emph{canonical} naming system for the natural numbers (see \citep{Shapiro1982}, \citep{Horsten2005}). Note that on this approach, standard numerals as well as efficient numerals qualify as adequate numeral systems. 
	
	One way to view the natural numbers is as the free algebra for a unary function and one generator. Addition and multiplication are, then, defined by recursion, where the possibility of such recursions is guaranteed by the fact that we have a free algebra. From this point of view, the standard or tally numbers indeed seem to have a preferred status: they are the standard syntactic representations of the elements of the free algebra. However, we would like to point out that the point of view of this specific free algebra is not the only way to think of the natural numbers. In a sense, viewing the natural numbers as this free algebra \emph{is} taking the tally numerals to be the basic numerals. Alternatively, we could take the dyadic representation to be primary and work with a theory with two successors. In such a context, the dyadic numerals would be the usual syntactic representations. Similarly, the notion of natural number as finite cardinal does not automatically dictate zero and successor to be the preferred signature.
	
	We will not deal with the question of the status of numerals in this paper and simply assume that both standard numerals as well as efficient numerals are adequate choices of a naming device. For a further discussion the reader is referred to \citep{Auerbach1994}.
	
	\subsection{Monotonicity and Self-Referentiality}
	\label{subsection:monotonicity}
	
	The notion of self-referential numberings can be generalised to arbitrary numeral functions.
	
	\begin{definition}
		A numbering $\xi$ is called \emph{self-referential for a numeral function $\nu$} if, for any formula $A(x)$, there exists $n \in \omega$ such that $n=\xi( A(\nu(n)) )$.
	\end{definition}
	
	As opposed to the notion of m-self-reference for sentences (see~Definition~\ref{def:mselfref}), self-reference for numberings is a mere technical concept and not intended to capture a pre-theoretical or philosophical notion. However, as we have seen, self-referential numberings immediately yield m-self-referential sentences.
	
	In the following sections we will investigate whether adequate numberings can be self-referential, i.e., whether self-referential numberings are indeed as contrived as typically believed. We will first examine whether monotonicity serves as an adequacy condition on numberings which is sufficiently restrictive to exclude self-referential numberings and will then turn to additional more restrictive adequacy constraints for numberings.
	
	Let $\mathsf{st} \colon \mathsf{ClTerm}_{\mathcal{L}} \to \omega$ be the function assigning to each closed $\mathcal{L}$-term the number of its proper sub-terms. We then get the following incompatibility result.
	
	\begin{lem}
		\label{lem:limitation:monotonicandSR}
		Let $\nu$ be a numeral function such that there is a constant $c \in \omega$ with $n - \mathsf{st}(\nu(n)) < c$ for all $n \in \omega$. Then, there is no numbering which is self-referential for $\nu$ and satisfies \textnormal{M2} and \textnormal{M3}.
	\end{lem}
	
	\begin{proof}
		Assume that $\xi$ is self-referential for $\nu$ and satisfies M2 and M3. Let $B(x)$ be an $\mathcal{L}$-formula in which $x$ does occur freely at least once. Set
		\[
		A(x) := \underbrace{B(x) \vee \ldots \vee B(x)}_{(c+1) \times}.
		\]
		Then there exists $n \in \omega$ such that $\xi(A(\nu(n))) = n$. By M2 and M3 we then get
		\begin{align*}
		n & = \xi(A(\nu(n)))\\
		& = \xi(\underbrace{B(\nu(n)) \vee \ldots \vee B(\nu(n))}_{(c+1) \times})\\
		& \geq \xi(\underbrace{B(\nu(n)) \vee \ldots \vee B(\nu(n))}_{c \times}) +1 \geq \cdots\\
		& \geq \xi(B(\nu(n))) + c\\
		& \geq \xi(\nu(n)) +c\\
		& \geq \mathsf{st}(\nu(n)) + c,
		\end{align*}
		in contradiction to the assumption that $n - \mathsf{st}(\nu(n)) < c$.
	\end{proof}
	
	Since $n - \mathsf{st}(\underline{n}) < 1$ for all $n \in \omega$, we conclude from the above lemma:
	
	\begin{cor}
		There is no self-referential numbering for standard numerals which is monotonic.
	\end{cor}
	
	After the second author's lecture on \emph{Cogwheels of Self-reference} at the workshop \emph{Ouroboros, Formal Criteria of Self-Reference in Mathematics and Philosophy} on February 17, 2018, Joel Hamkins asked whether the result on the non-monotonicity of self-referential G\"odel numberings also holds when we consider efficient numerals.
	
	We first note that Lemma~\ref{lem:limitation:monotonicandSR} does not apply to efficient numerals. We define:
	\begin{itemize}
		\item
		$\widetilde 0 := \zero$
		\item
		$\widetilde {n+1} := {\sf S}({\sf SS}\zero \times \widetilde n)$.
	\end{itemize}
	Let {\sf ev} be the evaluation function for arithmetical closed terms and let $|\alpha|$ denote the length of the string $\alpha$. We find:
	
	\begin{lemma}\label{lolligesmurf} For every $n \in \omega$
		\begin{itemize}
			\item ${\sf ev}(\widetilde n ) = 2^n -1$;
			\item $|\widetilde n| = 7n+1$;
			\item the number of subterms-qua-type of $\widetilde n$ is
			$\leq 2n+3$.
		\end{itemize}
	\end{lemma}
	
	From this lemma we conclude that the function $\lambda n.2^n -1$ grows exponentially while the number of sub-terms of $\overline{2^n -1}$ is not larger than $2n+3$ and thus only grows linearly. Hence, the assumption of Lemma~\ref{lem:limitation:monotonicandSR} cannot be satisfied for efficient numerals, since there is no constant $c$ such that $n - {\sf st}(\overline{n}) < c$ for all $n \in \omega$.
	
	Indeed, in what follows we show that the answer to Hamkins' question is \emph{no}, by constructing a monotonic numbering which is self-referential for efficient numerals.
	
	\section{A Monotonic Self-Referential Numbering}
	\label{section:construction2}
	Let $\mathcal L$ be an arithmetical language, as introduced in Subsection~\ref{subsection:language}, and
	let $\mathcal L(\mathsf{c})$ be $\mathcal L$ extended with a fresh constant $\mathsf{c}$. Let $A_0,A_1,\dots$ be an
	effective enumeration of all expressions of $\mathcal L(\mathsf{c})$. We assume that if $C \preceq A_n$, then, for some $k\leq n$, we have $C=A_k$.

	We define $\nusu A$ as the number of sub-expressions-qua-type of $A$, in other words, $\nusu A $ is the cardinality of
	$\verz{ B\mid B\preceq A}$. The following trivial observation is very useful:
	\begin{lemma}
		$\nusu {A_k} \leq k+1$.
	\end{lemma}
	
	\noindent Here is our construction.
	
	\[
	\fcolorbox{lightgray}{lightgray}{\begin{minipage}{33em}
		
		\medskip
		We construct a list $\Lambda :=(B_n)_{n\in \omega}$ in stages $k$. Let $\nub(k) := 2^{k+4}+1$.
		Each stage $k$ will result in a list $\Lambda_k = B_0,\dots, B_{\nub(k)-1}$.
		To simplify the presentation, we make $\Lambda_{-1}$ the empty list and $\nub(-1) := 0$.
		
		\medskip
		In stage $k$, we act as follows. Let $A_k^\ast := A_k [{\sf c} := {\sf S}\widetilde{(k+4)}]$.
		Let the sub-expressions of $A_k^\ast$ that do not occur in $\Lambda_{k-1}$ and are not of the form $\underline m$
		be $A_{i_0},\dots, A_{i_{\ell-1}}$,
		where the sequence $i_j$ is strictly increasing.
		We note that $\ell$ could be 0. We define $B_{\nub(k)-\ell +j} := A_{i_j}$, for $j <\ell $.
		Let $s$ be the smallest number such that $\underline s$ is not in  $\Lambda_{k-1}$.
		We set $B_{\nub(k-1)+p} := \underline{s+p}$, for $p< \nub(k)-\nub(k-1)-\ell$.
		\end{minipage}
	}
	\]
	
	\bigskip\noindent
	To see that our construction is well-defined it is sufficient that $\nub(k)-\ell \geq \nub(k-1)$ or, equivalently,
	$\ell \leq \nub(k)-\nub(k-1)$.
	We note that:
	\[
	\ell \leq \nusu {A_k^\ast} \leq \nusu{A_k}+\nusu{{\sf S}\widetilde {(k+4)}}-1  \leq k +1  + 2(k+4)+3 = 3 k + 12.
	\]
	(The $-1$ in the right-hand-side of the second inequality can be seen as follows. In case {\sf c} occurs in $A_k$ it is subtracted
	in the substitution. If {\sf c} does not occur, we have $A_k = A_k^\ast$ and, from this, the inequality follows immediately.)
	
	If $k=0$, we have $3\cdot 0  +12 = 12 < 17 = 2^{0+4}+1-0 = \nub(0) - \nub(-1)$.
	Let $k>0$. We find:
	\[3k+12  < 2^{k+3} = 2^{k+4}+1 - 2^{k+3}-1 = \nub(k)-\nub(k-1).\]

	\begin{lemma}\label{hippiesmurfA}
		Suppose $A_k \in \mathcal L$. Then, $A_k$ is in $\Lambda_k$ and, hence, in $\Lambda$.
	\end{lemma}
	
	\begin{proof}
		Consider stage $k$. We note that $A_k^\ast = A_k$. In case $A_k$ occurs in $\Lambda_{k-1}$, we are done.
		In case $A_k$ is not in $\Lambda_{k-1}$ and not of the form $\underline m$, clearly, $A_k$ will be added, so we are again done.
		Suppose $A_k = \underline m$, for some $m$. It follows that $m+1=\nusu{\underline m} \leq k+1$.
		We note that all sub-expressions of $A_k$ are all of the form $\underline {m}'$, for $m'\leq m$. So $\ell = 0$.
		Clearly, $m+1 \leq k+1 < \nub(k) - \nub(k-1)$. So, all sub-expressions of $A_k$ are either in $\Lambda_{k-1}$ or will be added.
	\end{proof}
	
	\begin{lemma}\label{neerslachtigesmurfA}
		Suppose $C \preceq B_n$. Then, for some $j\leq n$, we have $C = B_j$.
	\end{lemma}
	
	\begin{proof}
		Let $C \preceq B_n$.
		Suppose $B_n$ is added in stage $k$. Then, $B_n$ is either of the form $\underline m$ or a sub-expression 
		of $A^\ast_k$ not of the form $\underline m$.
		
		Suppose $B_n = \underline m$.
		Let $s$ be the smallest number such that $\underline s$ is not in  $\Lambda_{k-1}$.
		For all $s'< s$, we have $\underline s'\in \Lambda_{k-1}$. By our construction, all $\underline s''$, for
		$s\leq s''\leq m$ are added at stage $k$ in ascending order. So, for all $s^\circ \leq m$, we find that $\underline s^\circ$ precedes
		$\underline m$ in $ \Lambda_k$. Finally, each sub-expression $C$ of $\underline m$ is of the form $\underline s^\circ$, for
		some $s^\circ \leq m$.
		
		Suppose $B_n$ is added as a sub-expression of $A^\ast_k$ not of the form $\underline m$. 
		First suppose $C= \underline p$.
		We note that $p+1 = \nusu{\underline p} < \nusu{A_k} \leq k+1$. This tells us, by Lemma~\ref{hippiesmurfA}, that $\underline p$ is in $\Lambda_{k-1}$.
		Hence, it is added to $\Lambda$ before $B_n$.  Now suppose that $C$ is not of the form $\underline p$. 
		Then either $C$ is in $\Lambda_{k-1}$ or added to $\Lambda_k$ before $B_n$.
	\end{proof}
	
	\begin{lemma}
		The enumeration $\Lambda$ is without repetitions.
	\end{lemma}
	
	\begin{proof}
		Consider any $C \in \mathcal L$. If $C$ is not of the form $\underline m$ by our construction it will be only added once.
		
		Suppose $ C = \underline m$. We note that, by Lemma~\ref{neerslachtigesmurfA}, the $n$ such that $\underline n \in \Lambda_{k-1}$
		are downwards closed. This means that, in our construction, there is no $m\geq s$, such that $\underline m \in \Lambda_{k-1}$.
		So all the $\underline m$ that are added in any stage $k$ are new.
	\end{proof}
	
	\begin{lemma}\label{geestigesmurfA}
		Suppose ${\sf S}\widetilde n$ occurs in $\Lambda_k$, then $n \leq k+4$.
	\end{lemma}
	
	\begin{proof}
		We prove this by induction on $k$. In stage 0, we easily verify that the largest term of the form $\mathsf{S}\widetilde n$ can be at most
		$\mathsf{S} \widetilde 4$.
		
		Suppose, we have our desired estimate for $k-1$, where $k>0$. We prove our estimate for $k$.
		Clearly, the $\mathsf{S}\widetilde n$ occurring in $\Lambda_{k-1}$ do satisfy our estimate.
		Moreover, all the $\underline p$ added in stage $k$ do not provide new elements of the form $\mathsf{S}\widetilde n$.
		So, the only interesting case is the case where $A_k$ is not a standard numeral. Suppose it is not.
		The ${\sf S}\widetilde n$ that are added are all sub-expressions of $ A^\ast_k= A_k[{\sf c} := {\sf S}\widetilde{(k+4)}]$. 
		So, we should focus on the largest
		such sub-term.
		Suppose ${\sf S}\widetilde n \preceq A^\ast_k$.
		There are two possibilities:
		\begin{enumerate}[i.]
			\item ${\sf S}\widetilde n\preceq A_k$. In this case, we have 
			$n < \nusu{\mathsf{S}\widetilde n} \leq \nusu{A_k} \leq k+1 \leq k +4$.
			\item
			$n= k+4$, since the sub-term results from the substitution for {\sf c}.
		\end{enumerate}
		In each of the possible cases, we are done.
	\end{proof}
	
	\begin{lemma}
		\label{lem:gn3isselfref}
		Consider any $A$ such that {\sf c} occurs in $A$. Then, for some $n$, $B_n = A[{\sf c}:= \overline n]$.
	\end{lemma}
	
	\begin{proof}
		Suppose {\sf c} occurs in $A$. Let $A := A_k$. The only thing we have to show that $A^\ast_k$ is not in $\Lambda_{k-1}$.
		This is certainly true for $k=0$.
		Suppose $k>0$. We note that ${\sf S}\widetilde {(k+4)}$ occurs in $A^\ast_k$. So, it is sufficient to show that
		${\sf S}\widetilde {(k+4)}$ does not occur in $\Lambda_{k-1}$.
		By Lemma~\ref{geestigesmurfA}, whenever $\mathsf{S}\widetilde n$ occurs in $\Lambda_{k-1}$, we have
		$n \leq k+3$. We may conclude that
		$B_{\nub(k)-1} = A_k[\mathsf{c} := \mathsf{S}(\widetilde{k+4})]$. 
		Let $n := \nub(k)-1$. Since $\mathsf{S}(\widetilde{k+4}) = \overline{\nub(k)-1}$, we are done.
	\end{proof}
	
	Setting ${\sf gn}_0(A)$ to be the unique $n$ such that $A = B_n$, we have found a monotonic self-referential G\"odel numbering ${\sf gn}_0$,
	where monotonicity is defined with respect to the sub-expression relation.
	
	\section{Strings}
	\label{section:construction2B}
	Let $\mathcal L$ be an arithmetical language, as introduced in 
	Subsection~\ref{subsection:language}, and let $\mathcal A$ be the alphabet of $\mathcal L$. Let $\mathsf{c}$ be a fresh letter.
	We define:
	\begin{itemize}
		\item
		$\mathcal A^\ast$ is the set of all strings of $\mathcal A$, including the empty string $\varepsilon$. 
		\item
		$\mathcal A_{\sf c}$ is $\mathcal A$ extended with {\sf c}
		and $\mathcal A_{\sf c}^\ast$ is the set of all strings of $\mathcal A_{\sf c}$.\\
		(In $\mathcal L({\sf c})$ viewed as a subset of $\mathcal A_{\sf c}^\ast$, we treat {\sf c} as a constant-symbol.)
		\item
		$\alpha \sqsubseteq \beta$ iff $\alpha$ is a sub-string of $\beta$.
		\item
		$|\alpha|$ is the length of $\alpha$.
		\item
		$\nusus \alpha$ is the number of sub-strings-qua-type of $\alpha$, in other words, $\nusus \alpha $ is the cardinality of
		$\verz{ \beta\mid \beta\sqsubseteq \alpha}$. 
	\end{itemize}
	We have the following well-known fact.
	
	\begin{lemma}
		Suppose $|\alpha| = n$. Then, the number of non-empty-sub-string-oc\-cur\-ren\-ces in $\alpha$ is $\frac{n(n+1)}{2}$.
		As a consequence, $\nusus \alpha \leq \frac{n^2+n+2}{2}$.
	\end{lemma} 
	
	\begin{proof}
		Suppose $|\alpha| = n$. Let us number the spaces before and after the letter-occurrences in $\alpha$: $1,\dots,n+1$.
		Each non-empty-sub-string-occurrence corresponds uniquely to the pair of the space before and the space after the occurrence.
		So, the number of such occurrences is $\binom{n+1}{2} =  \frac{n(n+1)}{2}$.
	\end{proof}
	
	\noindent
	Since we are not striving for maximal efficiency, we will estimate $\nusus \alpha$ by $|\alpha|^2+1$. 
	Let $\mathfrak a := {\sf S}({\sf SS}\zero\times\noshow{)}$. Then, ${\sf S}\widetilde n = {\sf S}\mathfrak a^n\zero\noshow{(})^n$.
	So, $| {\sf S}\widetilde n | = 1+6n +1+n = 7n+2$ and, thus, $\nusus{{\sf S}\widetilde n} \leq (7n+2)^2+1 = 49 n^2+ 28 n +5$.
	
	\begin{remark}
		\label{remark:polish}
		We note that a much better estimate of  $\nusus{{\sf S}\widetilde n}$ is possible but that its growth rate will still be quadratic.
		In case we switch to Polish notation, this becomes linear.  Such notational choices like infix versus Polish are irrelevant in
		our construction and its verification when we are considering sub-expressions. In the sub-string format, they become active. 
		
		\medent
		{\footnotesize We provide the estimate in the Polish case. For the rest of this
			remark we work with Polish notation. As in Section~\ref{subsection:language}, we write {\sf M} for multiplication.
			
			Let $\mathfrak b := {\sf SMSS}\zero$. We have ${\sf S}\widetilde n ={\sf S} \mathfrak b^n \zero$. We note that, unlike $\mathfrak a$,
			the string $\mathfrak b$ stands for a meaningful entity, to wit $\lambda k\in \omega. 2k+1$. See \citep{viss:ambi11}. 
			We count $\nusus{{\sf S}\widetilde n}$ as follows. Clearly $\nusus{{\sf S}\widetilde 0} = \nusus{{\sf S} \zero} = 4$. Suppose $n>0$. We have: 
			\begin{itemize}
				\item
				The empty string: 1.
				\item
				The number of non-empty sub-strings of $\mathfrak b$: $\frac{5\cdot 6}{2} -2 = 13$. (We subtract 2, since {\sf S} occurs three times.)
				\item
				All strings of the form $\alpha \mathfrak b^i \beta$, where $i< n-1$, $\alpha$ is a non-empty final string of $\mathfrak b$ and $\beta$ is
				a non-empty initial string of $\mathfrak b$: $25(n-1)$.
				\item
				All sub-strings that contain the the final occurrence of $\zero$, except $\zero$ itself: $5n+1$. (We note that all these strings do not occur elsewhere, since they end with $\zero \zero$.)
				\item
				All sub-strings that contain the initial occurrence of {\sf S}, except  ${\sf S}\widetilde n$ itself, 
				{\sf S} and {\sf SS}: $5n-1$. (We note that the strings considered are unique
				since they start with {\sf SSM}. Only  ${\sf S}\widetilde n$ is an exception, since we already counted it in the previous case.)
			\end{itemize}
			So, \emph{in toto}, we have:  $\nusus{{\sf S}\widetilde n} = 1+13+25(n-1)+5n+1+5n-1 = 35n-11$.
			A similar estimate is also possible in the infix case, of course. However, the closing brackets will force the number to be quadratic in $n$.}
	\end{remark}
	
	We will use the following  insights.
	
	\begin{lemma}\label{grumpysmurf}
		Suppose  $\alpha \in\mathcal A^\ast_{\sf c}$.
		Suppose $m\neq 0$ and $n \neq 0$. Consider a string $\alpha \in\mathcal A^\ast_{\sf c}$ and suppose $\alpha$ contains overlapping occurrences $o$ of ${\sf S}\widetilde n$ and $p$ of ${\sf S} \widetilde{m}$. Then, $o=p$.
	\end{lemma}
	
	\begin{proof}
		Without loss of generality we may assume that $p$ starts at the same place as or at a later place than $o$. Since $o$ does not end with either {\sf S}, or {\sf SS}, the initial string 
		${\sf SS}(\noshow{)}$ of $p$ must lie in $o$. But the only sub-string of $o$ of the form 
		${\sf SS}(\noshow{)}$ is the intial one. Hence $o$ and $p$ start at the same place.
		By balance considerations, the matching closing brackets of the initial
		${\sf SS}(\noshow{)}$ in $o$ and $p$ should coincide. So, $o=p$  
	\end{proof} 
	
	\begin{lemma}
		Suppose $\alpha\in \mathcal A^\ast_{\sf c}$ and $\beta\in \mathcal A^\ast$, where {\sf c} occurs in $\alpha$ and $\beta$ is non-empty.
		Let $\nusus{\alpha} = p$, $\nusus{\beta}=q$ and $|\beta| = r$. Then,
		$\nusus{\alpha[{\sf c}:= \beta]} \leq (p-1)r^2+q$. 
		Since $q\leq r^2+1$, it follows that $\nusus{\alpha[{\sf c}:= \beta]} \leq pr^2+1$. 
	\end{lemma}
	
	\begin{proof}
		Let $\alpha^\ast$ be  $\alpha[{\sf c}:= \beta]$ where the original string $\alpha$ is supposed to be black and where we colored each sub-string occurrence
		of $\beta$ that resulted from the substitution with a unique color. We will consider black not to be a color here. 
		Consider any non-empty sub-string $\gamma^\ast$ of $\alpha^\ast$.  We replace all maximal sub-string occurrences of the same color by a single occurrence of {\sf c}.
		This corresponds to a unique sub-string occurrence of $\gamma$ in $\alpha$. If $\gamma$ is {\sf c}, it has $q-1$ originals-qua-type. 
		If $\gamma \neq{\sf c}$, the largest number of originals-qua-type obtains when $\gamma$ both begins and ends with a {\sf c}.
		This number of originals is $\leq r^2$. So, $\nusus{\alpha[{\sf c} := \beta]} \leq (p-1)r^2 + q-1 +1 =  (p-1)r^2 + q$.
	\end{proof}
	
	Let $\alpha_0,\alpha_1,\dots$ be an
	effective enumeration of the strings in $\mathcal A^\ast_{\sf c}$. We assume that if $\gamma \sqsubseteq \alpha_n$, then, 
	for some $k\leq n$, we have $\gamma= \alpha_k$.
	The following trivial observation is very useful:
	\begin{lemma}
		$\nusus {\alpha_k} \leq k+1$.
	\end{lemma}
	
	\noindent Here is our construction.
	
	\[
	\fcolorbox{lightgray}{lightgray}{\begin{minipage}{33em}
		
		\medskip
		We construct a list $\Lambda :=(\beta_n)_{n\in \omega}$ in stages $k$. Let $\nub(k) := 2^{k+15}+1$.
		Each stage $k$ will result in a list $\Lambda_k = \beta_0,\dots, \beta_{\nub(k)-1}$.
		To simplify the presentation, we make $\Lambda_{-1}$ the empty list and $\nub(-1) := 0$.
		
		\medskip
		In stage $k$, we act as follows. Let $\alpha_k^\ast := \alpha_k [{\sf c} := {\sf S}\widetilde{(k+15)}]$.
		Let the sub-strings of $\alpha_k^\ast$ that do not occur in $\Lambda_{k-1}$ and are not of the form $\bot^m$
		be $\alpha_{i_0},\dots, \alpha_{i_{\ell-1}}$,
		where the sequence $i_j$ is strictly increasing.
		We note that $\ell$ could be 0. We define $\beta_{\nub(k)-\ell +j} := \alpha_{i_j}$, for $j <\ell $.
		Let $s$ be the smallest number such that $\bot^s$ is not in  $\Lambda_{k-1}$.
		We set $\beta_{\nub(k-1)+p} := \bot^{s+p}$, for $p< \nub(k)-\nub(k-1)-\ell$.
		\end{minipage}
	}
	\]
	
	\bigskip\noindent
	To see that our construction is well-defined it is sufficient that $\nub(k)-\ell \geq \nub(k-1)$ or, equivalently,
	$\ell \leq \nub(k)-\nub(k-1)$. We note that, if $k=0$, we have $\alpha_{0} = \varepsilon$, and hence $\ell =1$.
	Moreover, $\nub(0)-\nub(0-1)= 2^{15}+1$. So, we are done. Suppose $k>0$. Clearly, $\nub(k)-\nub(k-1)= 2^{k+14}$. 
	In case {\sf c} does not occur in $\alpha_k$, we have $\ell \leq \nusus{\alpha_k^\ast}= \nusus{\alpha_k} \leq k+1 \leq 2^{k+14}$.
	In case {\sf c} does occur $\alpha_k$, we have: 
	\[
	\ell \leq \nusus{\alpha_k^\ast} \leq \nusus{\alpha_k}|{\sf S}\widetilde {k+15}|^2 + 1  \leq (k+1)(7(k+15)+2)^2+1 \leq 2^{k+14}
	\]
	We leave the verification of the last inequality to the diligent reader.

	\begin{lemma}\label{hippiesmurfB}
		Suppose $\alpha_k \in \mathcal A^\ast$. Then, $\alpha_k$ is in $\Lambda_k$ and, hence, in $\Lambda$.
	\end{lemma}
	
	\begin{proof}
		Consider stage $k$. We note that $\alpha_k^\ast = \alpha_k$. In case $\alpha_k$ occurs in $\Lambda_{k-1}$, we are done.
		In case $\alpha_k$ is not in $\Lambda_{k-1}$ and not of the form $\bot^m$, clearly, $\alpha_k$ will be added, so we are again done.
		Suppose $\alpha_k = \bot^m$, for some $m$. It follows that $m+1=\nusus{\bot^m} \leq k+1$.
		We note that all sub-strings of $\alpha_k$ are all of the form $\bot^{m'}$, for $m'\leq m$. So $\ell = 0$.
		Clearly, $m+1 \leq k+1 < \nub(k) - \nub(k-1)$. So, all sub-expressions of $\alpha_k$ are either in $\Lambda_{k-1}$ or will be added.
	\end{proof}
	
	\begin{lemma}\label{neerslachtigesmurfB}
		Suppose $\gamma \sqsubseteq \beta_n$. Then, for some $j\leq n$, we have $\gamma = \beta_j$.
	\end{lemma}
	
	\begin{proof}
		Let $\gamma\sqsubseteq \beta_n$.
		Suppose $\beta_n$ is added in stage $k$. Then, $\beta_n$ is either of the form $\bot^m$ or a sub-string
		of $\alpha^\ast_k$ not of the form $\bot^m$.
		
		Suppose $\beta_n = \bot^m$.
		Let $s$ be the smallest number such that $\bot^s$ is not in  $\Lambda_{k-1}$.
		For all $s'< s$, we have $\bot^{s'}\in \Lambda_{k-1}$. By our construction, all $\bot^{s''}$, for
		$s\leq s''\leq m$ are added at stage $k$ in ascending order. So, for all $s^\circ \leq m$, we find that $\bot^{s^\circ}$ precedes
		$\bot^m$ in $ \Lambda_k$. Finally,  $\gamma$ as a sub-string of $\bot^m$ is of the form $\bot^{s^\circ}$, for
		some $s^\circ \leq m$.
		
		Suppose $\beta_n$ is added as a sub-string of $\alpha^\ast_k$ not of the form $\bot^m$. 
		First suppose $\gamma= \bot^p$.
		We note that $p+1 = \nusus{\bot^p} < \nusus{\alpha_k} \leq k+1$. This tells us, by Lemma~\ref{hippiesmurfB}, that $\bot^p$ is in $\Lambda_{k-1}$.
		Hence, it is added to $\Lambda$ before $\beta_n$.  Now suppose that $\gamma$ is not of the form $\bot^p$. 
		Then, either $\gamma$ is in $\Lambda_{k-1}$ or added to $\Lambda_k$ before $\beta_n$.
	\end{proof}
	
	\begin{lemma}
		The enumeration $\Lambda$ is without repetitions.
	\end{lemma}
	
	\begin{proof}
		Consider any $\gamma \in \mathcal A^\ast$. If $\gamma$ is not of the form $\bot^m$ by our construction it will be only added once.
		Suppose $ \gamma = \bot^m$. We note that, by Lemma~\ref{neerslachtigesmurfB}, the $n$ such that $\bot^n \in \Lambda_{k-1}$
		are downwards closed. This means that, in our construction, there is no $m\geq s$, such that $\bot^m \in \Lambda_{k-1}$.
		So all the $\bot^m$ that are added in any stage $k$ are new.
	\end{proof}
	
	\begin{lemma}\label{geestigesmurfB}
		Suppose ${\sf S}\widetilde n$ occurs in $\Lambda_k$, then $n \leq k+15$.
	\end{lemma}
	
	\begin{proof}
		We prove this by induction on $k$. In stage 0, we easily verify that the largest term of the form $\mathsf{S}\widetilde n$ can be at most
		$\mathsf{S} \widetilde{15}$.
		
		Suppose, we have our desired estimate for $k-1$, where $k>0$. We prove our estimate for $k$.
		Clearly, the $\mathsf{S}\widetilde n$ occurring in $\Lambda_{k-1}$ do satisfy our estimate.
		Moreover, all the $\bot^p$ added in stage $k$ do not provide new elements of the form $\mathsf{S}\widetilde n$.
		So, the only interesting case is the case where $\alpha_k$ is not a string of $\bot$'s. Suppose it is not.
		The ${\sf S}\widetilde n$ that are added are all sub-strings of $ \alpha^\ast_k= \alpha_k[{\sf c} := {\sf S}\widetilde{(k+15)}]$. 
		So, we should focus on the largest
		such sub-string.
		Suppose ${\sf S}\widetilde n \sqsubseteq \alpha^\ast_k$.
		There are two possibilities:
		\begin{enumerate}[i.]
			\item ${\sf S}\widetilde n\sqsubseteq \alpha_k$. In this case, we have 
			$n < \nusus{\mathsf{S}\widetilde n} \leq \nusus{\alpha_k} \leq k+1 \leq k +15$.
			\item
			$n= k+15$, since the sub-string results from the substitution for {\sf c}.
		\end{enumerate}
		In each of the possible cases, we are done.
	\end{proof}
	
	\begin{lemma}
		Consider any $\alpha$ such that {\sf c} occurs in $\beta$. Then, for some $n$, $\beta_n = \alpha[{\sf c}:= \overline n]$.
	\end{lemma}
	
	\begin{proof}
		Suppose {\sf c} occurs in $\alpha$. Let $\alpha := \alpha_k$. The only thing we have to show is that $\alpha^\ast_k$ is not in $\Lambda_{k-1}$.
		This is certainly true for $k=0$.
		Suppose $k>0$. We note that ${\sf S}\widetilde {(k+15)}$ occurs in $\alpha^\ast_k$. So, it is sufficient to show that
		${\sf S}\widetilde {(k+15)}$ does not occur in $\Lambda_{k-1}$.
		By Lemma~\ref{geestigesmurfB}, whenever $\mathsf{S}\widetilde n$ occurs in $\Lambda_{k-1}$, we have
		$n \leq k+14$. We may conclude that
		$\beta_{\nub(k)-1} = \alpha_k[\mathsf{c} := \mathsf{S}(\widetilde{k+15})]$. 
		Since $\mathsf{S}(\widetilde{k+15}) = \overline{\nub(k)-1}$, we are done.
	\end{proof}
	
	Setting ${\sf gn}_1(\alpha)$ to be the unique $n$ such that $\alpha = \beta_n$, we have found a monotonic self-referential G\"odel numbering ${\sf gn}_1$.
	Here monotonicity is defined with respect to the sub-string relation.

	\section{G\"odel Numerals as Quotations}
	\label{section:numeralsasquotations}
	%%%%%%
	In the philosophical literature, G\"odel numerals are typically conceived of as canonical 
	names for their coded expressions. One of the most canonical 
	naming devices is quotation, assigning to each expression $A$ its name \enquote{$A$}. Accordingly, the Tarski biconditional $\mathsf{T}(\underline{\xi(A)}) \leftrightarrow A$ for instance is often taken as an arithmetical proxy for the sentence \enquote{\enquote{$A$} is true iff $A$}.
	
	The standard numeral $\underline{\xi(A)}$ is \textit{prima facie} not the only candidate for an arithmetical proxy for quotation. As there are different adequate ways to quote an expression $A$, such as \enquote{$A$}, $\mlq A \mrq$, \guillemotleft $A$\guillemotright, etc., different numerals may be reasonably taken as adequate arithmetical proxies of $A$'s quotation, such as $\underline{\xi(A)}$, $\overline{\xi(A)}$, etc.\footnote{\label{ftn:numeralasquotation}While we agree that both standard and efficient G\"odel numerals are canonical names for their coded expressions, we remain sceptical as to whether they should be understood as proxies for quotations. For instance, the word \enquote{snow} can be adequately named by its structural-descriptive name \textit{ess} $\concat$ \textit{en} $\concat$ \textit{oh} $\concat$ \textit{doubleyoo}, by \enquote{\enquote{wons}}, by \enquote{\enquote{ssnnooww}}, by baptising it \enquote{(1)}, etc., all of which do not contain the named expression. The justification of the constraint M4($\nu$), introduced below, seems thus to rely on a specific feature of quotation which is not shared by several other adequate naming devices.}
	
	Once $\nu$-numerals are taken as adequate arithmetical proxies for quotations, another criterion for numberings may be extracted from the basic idea underlying monotonicity, which has been suggested by \cite{Halbach2018}. Recall that monotonic numberings have the property that the code of an expression $A$ is larger than the code of any expression (strictly) contained in~$A$. Since each quotation properly contains its quoted expression, it may be argued, along the lines of Section~\ref{subsection:adequacy}, that adequate numberings should preserve this structural feature by requiring that the code of the G\"odel $\nu$-numeral of an expression $A$ is larger than the code of $A$ itself:
	\begin{enumerate}
		\item[M4($\nu$).] $\xi(A) < \xi( \nu(\xi(A)) )$ for all $A \in \mathcal{L}$,
	\end{enumerate}
	
	We call a numbering $\xi$ of $\mathcal{L}$ \textit{strongly monotonic for $\nu$-numerals}, if $\xi$ satisfies M1-M3 as well as M4($\nu$).
	
	We note that M4($\nu$) rules out that $\xi(\zero) = 0$ and $\nu(0) = \zero$, which holds for certain length-first orderings and standard numerals. These standard numberings are not excluded by its weakened version M4$^\ast$($\nu$), which is obtained by replacing $<$ by $\leq$ in M4($\nu$). In what follows, we will always employ the version of M4($\nu$) which yields the stronger result.
	
	Since the standard numeral $\underline{n}$ has $n$-many proper sub-terms, any numbering satisfying M1 also satisfies M4($\underline{\cdot}$). Strong monotonicity is however not entailed by monotonicity when efficient numerals are employed. For instance, as will follow from the next lemma, the numbering ${\sf gn}_0$ is monotonic but not strongly monotonic for efficient numerals. In fact, strong monotonicity is incompatible with the self-referentiality of numberings.
	
	\begin{lem}
		\label{lem:limitation1}
		Let $\nu$ be a numeral function for $\mathcal{L}$. There is no numbering of $\mathcal{L}$ which is self-referential for $\nu$ and satisfies \textnormal{M3 \& M4$^\ast$($\nu$)}.
	\end{lem}
	
	\begin{proof}
		Suppose $\xi$ is a self-referential numbering of $\mathcal{L}$. Consider any formula $A(x)$ with free variable $x$ and $n \in \omega$ such that $\xi(A(\nu(n))) = n$. Since $\nu(n) \prec A(\nu(n))$, using M3 and M4$^\ast$($\nu$) 
		yields the contradiction\qedright
		\begin{equation*}
		n = \xi(A(\nu(n))) \leq \xi(\nu(\xi(A(\nu(n))))) = \xi(\nu(n)) < \xi(A(\nu(n))) = n.
		\end{equation*}
	\end{proof}
	
	Thus, if we strengthen the conception of monotonicity in the above manner, then, by Halbach's standards, there are no reasonable numberings which are self-referential.
	
	\subsection{A Critical Remark}
	
	We are somewhat sceptical whether strong monotonicity really is a necessary constraint for reasonable numberings (see also footnote~\ref{ftn:numeralasquotation}). For instance, we now show that the length-first numbering is not strongly monotonic for a fairly reasonable numeral function based on prime decomposition. Any philosophically stable view which excludes self-referential numberings as adequate formalisation choices along the above lines, hence appears to be committed to render the introduced numeral function below an inadequate naming device. 
	
	\begin{example}	\label{prideco}	
		Let $\mathcal{L}^{\sf Q} := \mathcal{L}^0 \cup \{ {\sf Q} \}$, where ${\sf Q}$ is a unary function symbol for \textit{taking the square}. Let $\mathfrak g$ be a numbering of $\mathcal{L}^{\sf Q}$ based on the length-first ordering.
		
		Are there plausible numerals for which strong monotonicity fails w.r.t.\ $\mathfrak g$? This is indeed the case.
		We give an example of such numerals that seems reasonably natural.
		
		We consider the following numeral function {\sf pd}. (`{\sf pd}' stands for \emph{prime decomposition}.)
		\begin{enumerate}[i.]
			\item ${\sf pd}(0) := \zero$ and ${\sf pd}(1) := {\sf S}\zero$.
			\item Suppose $n > 1$, and $n$ is not a power of a prime. Let $p$ be the smallest prime
			that divides $n$. Say $n = p^i\cdot m$, where
			$p$ does not divide $m$ and $i>0$. Then, ${\sf pd}(n) := \times {\sf pd}(p^i){\sf pd}(m)$.
			\item Suppose $n = p^i$, where $p$ is prime and $i>0$. We take  ${\sf pd}(p) := \underline p$. In case $i = 2k$, for $k>0$,
			we set ${\sf pd}(n) := {\sf Q}\,{\sf pd}(p^k)$. In case $i = 2k+1$, for $k>0$, we set 
			${\sf pd}(n) := \times \underline p\, {\sf Q}\,{\sf pd}(p^k)$.
		\end{enumerate}
		We find that, e.g., \[{\sf pd}(2^{2^m}) =  \overbrace{{\sf Q} \dots {\sf Q}}^{m \times}\underline 2.\]
		It follows that $|{\sf pd}(2^{2^m})| = m + 3$. Hence, by Lemma~\ref{lem:lengthfirstestimates},			
		\[\mathfrak g({\sf pd}(2^{2^m})) \leq \frac{16^{m+4}-1}{15} \]
		Since $2^{2^m}$ has double exponential growth, the failure of strong monotonicity is immediate.
		
		We note that {\sf pd} has good properties for multiplication but does not seem to be very friendly towards successor and addition.
		
		One can imagine many variants of {\sf pd}. For example, we can replace $\underline p$ by $\overline p$ in the definition.
		Alternatively, we can make $\mathfrak{p}$ where we replace clause (iii) by:
		\begin{itemize}
			\item Suppose $n = p^i$, where $p$ is prime and $i>0$. We take  ${\mathfrak p}(p) := {\sf S}{\mathfrak p}(p-1)$. In case $i = 2k$, for $k>0$,
			we set $\mathfrak p(n) := {\sf Q}\,\mathfrak p(p^k)$. In case $i = 2k+1$, for $k>0$, we set 
			$\mathfrak p(n) := \times {\sf S}\mathfrak p(p-1)\,{\sf Q}\,\mathfrak p(p^k)$.
		\end{itemize}
	\end{example}
	
	We leave an in-depth discussion of this constraint's adequacy for another occasion and, from now on, simply assume that there are good reasons to consider strong monotonicity to be a necessary constraint for reasonable numberings. In the next section we further examine the implications of this constraint on the attainability of self-reference.
	
	\subsection{A Strongly Monotonic Numbering with Strong Diagonalisation}
	\label{subsection:gn6}
	
	While strong monotonicity excludes self-referential numberings (Lemma~\ref{lem:limitation1}), we now show that this constraint is not sufficiently restrictive to rule out m-self-reference in $\mathcal{L}^0$. In particular, we construct a strongly monotonic numbering ${\sf gn}_2$ for efficient numerals which satisfies the Strong Diagonal Lemma for $\mathcal{L} \supseteq \mathcal{L}^0$:
	
	\begin{lem}[Strong Diagonal Lemma for ${\mathcal{L}}$]
		\label{lem:strongdiaglem}
		Let $C$ be an $\mathcal{L}$-expression with $x$ as a free variable. Then there exists a closed term $t$ in $\mathcal{L}$ such that ${\mathsf{R}_{\mathcal{L}} \vdash t = \underline{{\sf gn}_2(C[x := t])}}$.
	\end{lem}
	
	\noindent
	We note that we escape Lemma~\ref{lem:limitation1} here by allowing $t$ to be different
	from the efficient numeral of the value of $t$.
	
	Let $\mathcal L$ be an arithmetical language, as introduced in Subsection~\ref{subsection:language}, and
	let $\mathcal L(\mathsf{c})$ be $\mathcal L$ extended with a fresh constant $\mathsf{c}$. Let $A_0,A_1,\dots$ be an
	effective enumeration of all expressions of $\mathcal L(\mathsf{c})$. We assume that if $C \preceq A_n$, then, for some $k\leq n$, we have $C=A_k$.
	
	We define $\nusu A$ as the number of subexpressions-qua-type of $A$, in other words, $\nusu A $ is the cardinality of
	$\verz{ B\mid B\preceq A}$. The following trivial observation is very useful:
	\begin{lemma}
		$\nusu {A_k} \leq k+1$.
	\end{lemma}
	
	\noindent
	Let the function $\widehat{\cdot} \colon \omega \to \mathcal{L}$ be given by setting $\widehat{0} := \mathsf{v}$ and $\widehat{n+1} := \widehat{n} '$. 
	As in the case of standard numerals, expressions of the form $\widehat m$ are downwards closed with regard to the sub-expression relation and 
	we have that $\nusu{\widehat m} = m+1$. In the following construction, we use expressions of the form $\widehat m$ instead of 
	standard numerals as fillers. This ensures that the constructed enumeration does not start with $\overline 0$ and thus that the resulting numbering is strongly monotonic.
	
	Here is our construction.
	
	\[
	\fcolorbox{lightgray}{lightgray}{\begin{minipage}{33em}We construct a list $\Lambda :=(B_n)_{n\in \omega}$ in stages $k$. Let $\nub(k) := 2^{2k+4}+1$.
		Each stage $k$ will result in a list $\Lambda_k = B_0,\dots, B_{\nub(k)-1}$.
		To simplify the presentation, we make $\Lambda_{-1}$ the empty list and $\nub(-1) := 0$.
		
		\medskip
		In stage $k$, we act as follows. Let $A_k^\ast := A_k [{\sf c} := {\sf S}\widetilde{(k+2)}]$.
		Let the sub-expressions of $A_k^\ast$ that do not occur in $\Lambda_{k-1}$ and are not of the form $\widehat m$
		be $A_{i_0},\dots, A_{i_{\ell-1}}$,
		where the sequence $i_j$ is strictly increasing.
		We note that $\ell$ could be 0. We define $B_{\nub(k)-\ell +j} := A_{i_j}$, for $j <\ell $.
		Let $s$ be the smallest number such that $\widehat s$ is not in  $\Lambda_{k-1}$.
		We set $B_{\nub(k-1)+p} := \widehat{s+p}$, for $p< \nub(k)-\nub(k-1)-\ell$.
		\end{minipage}
	}
	\]
	
	\bigskip\noindent
	To see that our construction is well-defined it is sufficient that $\nub(k)-\ell \geq \nub(k-1)$ or, equivalently,
	$\ell \leq \nub(k)-\nub(k-1)$.
	We note that:
	\begin{equation}
	\label{equ:boundonell}
	\tag{$\ast$}
	\ell \leq \nusu {A_k^\ast} \leq \nusu{A_k}+\nusu{{\sf S}\widetilde {(k+2)}}-1  \leq k +1  + 2(k+2)+3 = 3 k + 8.
	\end{equation}
	(The $-1$ in the right-hand-side of the second inequality can be seen as follows. In case {\sf c} occurs in $A_k$ it is subtracted
	in the substitution. If {\sf c} does not occur, we have $A_k = A_k^\ast$ and, from this, the inequality follows immediately.)
	
	If $k=0$, we have $3\cdot 0  +8 = 8 < 17 = 2^{2(0+2)}+1-0 = \nub(0) - \nub(-1)$.
	Let $k>0$. We find:
	\[3k+8  < 3 \cdot 2^{2k+2} = 2^{2k+4}+1 - 2^{2k+2}-1 = \nub(k)-\nub(k-1).\]

	\begin{lemma}\label{lem:AkinLambdak}
		Suppose $A_k \in \mathcal L$. Then, $A_k$ is in $\Lambda_k$ and, hence, in $\Lambda$.
	\end{lemma}
	
	\begin{proof}
		Consider stage $k$. We note that $A_k^\ast = A_k$. In case $A_k$ occurs in $\Lambda_{k-1}$, we are done.
		In case $A_k$ is not in $\Lambda_{k-1}$ and not of the form $\widehat m$, clearly, $A_k$ will be added, so we are again done.
		Suppose $A_k = \widehat m$, for some $m$. It follows that $m+1=\nusu{\widehat m} \leq k+1$.
		We note that all sub-expressions of $A_k$ are of the form $\widehat {n}$, for $n\leq m$. So $\ell = 0$.
		Clearly, $m+1 \leq k+1 < \nub(k) - \nub(k-1)$. So, all sub-expressions of $A_k$ are either in $\Lambda_{k-1}$ or will be added.
	\end{proof}
	
	\begin{lemma}\label{lem:monoton}
		Suppose $C \preceq B_n$. Then, for some $j\leq n$, we have $C = B_j$.
	\end{lemma}
	
	\begin{proof}
		Let $C \preceq B_n$.
		Suppose $B_n$ is added in stage $k$. Then, $B_{n}$ is either of the form $\widehat m$ or a sub-expression of $A^\ast_k$ not of the form $\widehat m$.
		
		Suppose $B_n = \widehat m$.
		Let $s$ be the smallest number such that $\widehat s$ is not in  $\Lambda_{k-1}$.
		For all $u< s$, we have $\widehat u\in \Lambda_{k-1}$. By our construction, all $\widehat v$, for
		$s\leq v \leq m$ are added at stage $k$ in ascending order. So, for all $w \leq m$, we find that $\widehat w$ precedes
		$\widehat m$ in $ \Lambda_k$. Finally, each sub-expression $C$ of $\widehat m$ is of the form $\widehat w$, for
		some $w \leq m$.
		
		Suppose $B_n$ is added as a sub-expression of $A^\ast_k$ not of the form $\widehat m$. First suppose $C= \widehat p$.
		We note that $p+1 = \nusu{\widehat p} < \nusu{A_k} \leq k+1$. This tells us, by Lemma~\ref{lem:AkinLambdak},
		that $\widehat p$ is in $\Lambda_{k-1}$.
		Hence, it is added to $\Lambda$ before $B_n$.  Now suppose that $C$ is not of the form $\widehat p$. 
		Then, either $C$ is in $\Lambda_{k-1}$ or added to $\Lambda_k$ before $B_n$.
	\end{proof}
	
	\begin{lemma}
		The enumeration $\Lambda$ is without repetitions.
	\end{lemma}
	
	\begin{proof}
		Consider any $C \in \mathcal L$. If $C$ is not of the form $\widehat m$ by our construction it will be only added once.
		
		Suppose $ C = \widehat m$. We note that, by Lemma~\ref{lem:monoton}, the $n$ such that $\widehat n \in \Lambda_{k-1}$
		are downwards closed. This means that, in our construction, there is no $m\geq s$, such that $\widehat m \in \Lambda_{k-1}$.
		So all the $\widehat m$ that are added in any stage $k$ are new.
	\end{proof}
	
	\begin{lemma}\label{lem:Stermupperboundonk}
		Suppose ${\sf S}\widetilde n$ occurs in $\Lambda_k$, then $n \leq k+2$.
	\end{lemma}
	
	\begin{proof}
		We prove this by induction on $k$. In stage 0, we easily verify that the largest term of the form $\mathsf{S}\widetilde n$ can be at most
		$\mathsf{S} \widetilde 2$.
		
		Suppose, we have our desired estimate for $k-1$, where $k>0$. We prove our estimate for $k$.
		Clearly, the $\mathsf{S}\widetilde n$ occurring in $\Lambda_{k-1}$ do satisfy our estimate.
		Moreover, all the $\widehat p$ added in stage $k$ do not provide new elements of the form $\mathsf{S}\widetilde n$.
		So, the only interesting case is the case where $A_k$ is not of the form $\widehat p$. Suppose it is not.
		The ${\sf S}\widetilde n$ that are added are all sub-expressions of $ A^\ast_{k}= A_k[{\sf c} := {\sf S}\widetilde{(k+2)}]$. So, we should focus on the largest
		such sub-term.
		Suppose ${\sf S}\widetilde n \preceq A^\ast_k$.
		There are two possibilities:
		\begin{enumerate}[i.]
			\item ${\sf S}\widetilde n\preceq A_k$. In this case, we have $n < \nusu{\mathsf{S}\widetilde n} \leq \nusu{A_k} \leq k+1 \leq k +2$.
			\item
			$n= k+2$, since the sub-term results from the substitution for {\sf c}.
		\end{enumerate}
		In each of the possible cases, we are done.
	\end{proof}
	
	Setting ${\sf gn}_2(A)$ to be the unique $n$ such that $A = B_n$, we have found a monotonic G\"odel numbering ${\sf gn}_2$. 
	
	In order to show that ${\sf gn}_2$ satisfies M4($\overline{\cdot}$), i.e., that ${\sf gn}_2$ is strongly monotonic for efficient numerals, we prove the following auxiliary lemma.
	
	\begin{lemma}
		\label{lem:auxlem1}
		Suppose $\overline{p} \preceq A_k[\mathsf{c} := \overline r]$, for any $p,r \in \omega$. Then,
		$p < 2^{k}(r+2)$.
	\end{lemma}
	
	\begin{proof}
		Let $p,r \in \omega$ be given such that $\overline{p} \preceq A_k[\mathsf{c} := \overline r]$. Then (i) $\overline{p} \preceq A_k$ or (ii)
		$\overline{p}\preceq \overline r$ or (iii) there exists $t \preceq A_k$ containing 
		$\mathsf{c}$ such that $\overline p = t[\mathsf{c} := \overline r]$. In case (i), we have, $p < 2^{\nusu{\overline p}} \leq 2^{\nusu{A_k}} \leq 2^{k+1} \leq 2^{k}(r+2)$. In case (ii), we have $p\leq r$.
		So both in Case (i) and (ii), we find $p< 2^{k}(r+2)$.
		
		Suppose we are are in case (iii).
		Since $\overline{r}$ and $t[\mathsf{c} := \overline r]$ are efficient numerals, there exist $k,j$ such that $\overline{r}$ and $t$ are of the following forms:
		\begin{itemize}
			\item $\overline{r} = \mathsf{S}_{a_0}( \cdots \mathsf{S}_{a_k}(\zero) \cdots )$;
			\item $t = \mathsf{S}_{c_0}( \cdots \mathsf{S}_{c_j}(\mathsf{c}) \cdots )$;
		\end{itemize}
		where $a_p, c_q \in \{ 1, 2 \}$ for each $p \leq k$ and $q \leq j$, and $\mathsf{S}_1(x) := \mathsf{S} (\mathsf{S}\mathsf{S} \mathsf{0} \times x)$ and 
		$\mathsf{S}_2(x) := \mathsf{S} \mathsf{S} (\mathsf{S}\mathsf{S} \mathsf{0} \times x)$. Moreover, $\mathsf{ev}(\overline{r}) = \sum^k_{i = 0} a_i b^i$. Since
		\[
		t(\overline{r}) = \mathsf{S}_{c_0}( \cdots \mathsf{S}_{c_j}(\mathsf{S}_{a_0}( \cdots \mathsf{S}_{a_k}(\mathsf{0}) \cdots )) \cdots )
		\]
		we have
		\begin{align*}
		p & = \mathsf{ev}(t(\overline{r}))\\
		& = \sum^j_{i = 0} c_i 2^i + \sum^k_{i = 0} a_i 2^{i + j +1}\\
		& = \sum^j_{i = 0} c_i 2^i + 2^{j +1} \sum^k_{i = 0} a_i 2^{i}\\
		&  < 2^{j+2} + 2^{j +1} r\\
		& = 2^{j+1}(r+2).
		\end{align*}
		For each $i \leq j$, the expression $\mathsf{S}_{c_i}( \cdots \mathsf{S}_{c_{j}}(\mathsf{c}) \cdots )$ is a subterm of $A_k$. Since also $\mathsf{c}$ is a subterm of $A_k$, we have $j+2 \leq \nusu{A_k} \leq k+1$. We thus conclude that $p < 2^{k}(r+2)$.
	\end{proof}
	
	We now show that ${\sf gn}_2$ satisfies M4($\overline{\cdot}$).
	
	\begin{lemma}
		\label{lem:gn4isstronglymonotonic}
		$m < {\sf gn}_2( \overline{m} )$ for all $m \in \omega.$
	\end{lemma}
	
	\begin{proof}
		Let $\overline m$ be added in stage $k$, i.e., $\overline m \in \Lambda_k \setminus \Lambda_{k-1}$. We then have
		$B_{\nub(k) - \ell + j} = \overline{m}$, for some $j < \ell$. (We use here the fact that no filler is of the form $\overline{m}$). We have $\ell \leq 3k + 8$, as noted in (\ref{equ:boundonell}) subsequent to presenting our construction. Hence,
		\begin{equation*}
		2^{2k+4}  - 3k - 7 = \nub(k) - (3k + 8) \leq \nub(k) - \ell \leq {\sf gn}_2( \overline{m} ).
		\end{equation*}
		We moreover have $\overline{m} \preceq A_k [{\sf c} := {\sf S}\widetilde{(k+2)}]$. Since ${\sf S}\widetilde{(k+2)} = \overline{2^{k+2}}$ we get
		\begin{equation*}
		m < 2^{k} (2^{k+2}+2) =  3 \cdot 2^{k+1},
		\end{equation*}
		by Lemma~\ref{lem:auxlem1}. It is easy to check that $3 \cdot 2^{k+1} < 2^{2k+4}  - 3k - 7$. Hence, $m < {\sf gn}_2( \overline{m} )$.
	\end{proof}
	
	Even though ${\sf gn}_2$ cannot be self-referential for efficient numerals by Lemma~\ref{lem:limitation1}, the following modification holds.
	
	\begin{lem}
		\label{lem:prediaglem}
		Consider any $A$ such that {\sf c} occurs in $A$. Then there exists $n$ such that ${\sf gn}_2(A[{\sf c}:= \overline n]) = n^2$.
	\end{lem}
	
	\begin{proof}
		Suppose {\sf c} occurs in $A$. Let $A := A_k$. The only thing we have to show is that $A^\ast_k$ is not in $\Lambda_{k-1}$.
		This is certainly true for $k=0$.
		Suppose $k>0$. We note that ${\sf S}\widetilde {(k+2)}$ occurs in $A^\ast_k$. So, it is sufficient to show that
		${\sf S}\widetilde {(k+2)}$ does not occur in $\Lambda_{k-1}$.
		By Lemma~\ref{lem:Stermupperboundonk}, whenever $\mathsf{S}\widetilde n$ occurs in $\Lambda_{k-1}$, we have
		$n \leq k+1$. We may conclude that
		$B_{\nub(k)-1} = A_k[\mathsf{c} := \mathsf{S}(\widetilde{k+2})]$.
		Let $n := 2^{k+2}$. Since $n^2 = \nub(k)-1$ and $\mathsf{S}(\widetilde{k+2}) = \overline{n}$, we are done.
	\end{proof}
	
	We now can immediately derive the Strong Diagonal Lemma.
	
	\begin{proof}[Proof of Lemma~\ref{lem:strongdiaglem}.]
		Let $C(x)$ be given and set $A := C[x := \mathsf{c} \times \mathsf{c}]$. By Lemma~\ref{lem:prediaglem} there exists $n \in \omega$ such that ${\sf gn}_2(A[\mathsf{c} := \overline{n}]) = n^2$. Set $t := \overline{n} \times \overline{n}$. We then have $A[\mathsf{c} := \overline{n}] = C[x := t]$. Moreover $\mathsf{R}_{\mathcal{L}} \vdash t = \underline{n^2}$. Hence ${\mathsf{R}_{\mathcal{L}} \vdash t = \underline{{\sf gn}_2(C[x := t])}}$.
	\end{proof}
	
	By suitably adapting the above construction, we obtain a numbering which is monotonic with respect to the sub-string relation, satisfies M4($\overline{\cdot}$) and provides the existence of m-self-referential sentences for $\mathcal{L}$ (i.e., satisfies Lemma~\ref{lem:strongdiaglem}). We leave the details to the diligent reader.
	
	\section{Computational Constraints}
	\label{section:compadequacy}
	\newcommand{\pol}{\mathfrak{Po}}
	\newcommand{\ele}{\mathfrak{El}}
	\newcommand{\prc}{\mathfrak{Pr}}
	In addition to requiring monotonicity, other adequacy constraints for G\"odel numberings 
	can be extracted from the literature. Let $\pol$, $\ele$ and $\prc$ denote the classes of p-time, 
	(Kalm\'ar) elementary and primitive recursive functions respectively. As usual, these classes can be 
	extended to relations and partial functions on $\omega$ as follows.
	
	\begin{definition}
		Let $\mathcal{C} \in \{\pol,\ele,\prc\}$. We say that a subset of $\omega^k$ or a relation on $\omega$ is in $\mathcal{C}$, if its characteristic function is in $\mathcal{C}$.
		
		Let $R \subseteq \omega^k$ be given. We say that a function $f \colon R \to \omega$ is in $\mathcal{C}$, if the total function $f' \colon \omega^k \to \omega$ given by
		\[
		f'(m_1, \ldots, m_k) := 
		\begin{cases}
		f(m_1, \ldots, m_k) +1 & \text{if } m_1, \ldots, m_k \in R;\\
		0 & \text{otherwise.}
		\end{cases}	
		\]
		is in $\mathcal{C}$.
	\end{definition}
	
	Reasonable numberings are commonly required to represent certain syntactic relations and operations by \textit{primitive recursive} relations and operations on $\omega$ (see \citep[Section 5.1]{Halbach2014}). The usual definitions of the syntactic relations (i.e., their extensions) $\mathsf{Var}_{\mathcal{L}}$ (\enquote{is a variable}), $\mathsf{Ter}_{\mathcal{L}}$ (\enquote{is a term}), $\mathsf{AtFml}_{\mathcal{L}}$ (\enquote{is an atomic formula}), $\mathsf{Fml}_{\mathcal{L}}$ (\enquote{is an a formula}), $\mathsf{Sent}_{\mathcal{L}}$ (\enquote{is a sentence}), \enquote{is a free variable in}, \enquote{is a $\mathsf{PA}$-derivation of} etc.,
	are either explicit or of simple recursive structure. Adequate numberings may be required to preserve this simple algorithmic nature, in virtue of representing these relations by primitive recursive relations on $\omega$. More precisely, for any adequate numbering $\xi$, the characteristic functions of the sets of $\xi$-codes of $\mathsf{Var}_{\mathcal{L}}$, $\mathsf{Ter}_{\mathcal{L}}$, etc.\ are required to be primitive recursive. A similar constraint can be extracted for the syntactic functions $\dot{=}$, $\dot\neg$, $\dot\wedge$, $\dot\vee$, $\dot\forall$, $\dot\exists$, $\dot{\mathsf{S}}$, $\dot{+}$, $\dot\times$, $\mathsf{Sub}$ etc., where for instance, $\dot{\wedge} \colon \mathsf{Fml}_{\mathcal{L}}^2 \to \mathsf{Fml}_{\mathcal{L}}$ maps a pair $\langle A,B \rangle$ of formul{\ae} to the conjunction $A \wedge B$, $\dot{\forall} \colon \mathsf{Fml}_{\mathcal{L}} \times \mathsf{Var}_{\mathcal{L}} \to \mathsf{Fml}_{\mathcal{L}}$ maps a pair $\langle A,x \rangle$ to the universal formula $\forall x A$, etc., and where $\mathsf{Sub}(A,x,t)$ is the result of substituting the term $t$ for the variable $x$ in the formula~$A$. Once again, adequate numberings may be required to represent these syntactic functions by primitive recursive functions on~$\omega$. Finally, the standard and the efficient numeral function may be required to be represented primitive recursively.
	
	The arithmetisation	of syntax in weaker theories, like Buss' ${\sf S}^1_2$ and $\mathrm I\Delta_0+\Omega_1$, requires representation of the syntactic relations and functions in weaker theories.\footnote{See \citep{pudl:cuts85} for ways to deal with inefficient representations 
		in weak theories.} A closer inspection of what is going on shows that, in the presence of a reasonable G\"odel numbering like the one based on the length-first ordering, the syntactic relations and functions can be made p-time \emph{without extra effort}.
	See \citep{buss:boun86} and \citep{haje:meta91}. 
	The G\"odel numberings we construct here fit most naturally in the intermediate function class
	$\ele$.
	
	\begin{definition} Let $\mathcal{C} \in \{\pol,\ele,\prc\}$ and let $\nu$ be a numeral function. We say that a numbering $\xi$ of $\mathcal{L}$ is \textit{$\mathcal{C},\nu$-adequate}, if
		\begin{enumerate}[1.]
			\item the set of $\xi$-codes of each syntactic relation specified above is in $\mathcal{C}$;
			\item the \textit{$\xi$-tracking function} of each syntactic function specified above is in $\mathcal{C}$, i.e., for each $k$-ary syntactic function $f$, the function $\xi(f)$ given by $\xi(f)(m_1, \ldots, m_k) := \xi(f(\xi^{-1}(m_1), \ldots,  \xi^{-1}(m_k) ) )$ is in $\mathcal{C}$;
			\item the function  $\xi \circ \nu$ is in $\mathcal C$.
		\end{enumerate}
		We say that $\xi$ is \emph{$\mathcal C$-adequate} iff 
		($\mathcal C\in \verz{\ele,\prc}$ and  $\xi$ is both
		$\mathcal C,(\overline\cdot)$-adequate and $\mathcal C,(\underline\cdot)$-adequate)
		or ($\mathcal C = \pol$ and $\xi$ is $\mathcal C,(\overline\cdot)$-adequate).
	\end{definition}
	
	The definition of \emph{$\mathcal C$-adequate} is somewhat awkward since, for a good numbering like the one based on the length-first ordering, the function $\xi\circ(\underline \cdot)$
	will be exponential. This is the reason for the use of efficient numerals in the context of weak theories. It would be interesting, and conceivably genuinely useful,
	to explore whether there are $\pol$-adequate 
	numberings $\xi$ for which
	$\xi\circ(\underline \cdot)$ is p-time.
	
	\begin{remark}
		Feferman's 
		G\"odel numbering, say {\sf fef}, in \citep{fefe:arit60} 
		is a good example to reflect on. The tracking functions are
		all elementary.
		The code for 0 is 3 and the
		code for successor is 7. We have ${\sf fef}({\sf S}t) =
		2^7\cdot 3^{{\sf fef}(t)}$. So, the numeral function for standard numerals will be superexponential and not elementary. Similarly, for efficient numerals.  
		
		Suppose we replace Feferman's coding for the terms by an efficient coding but leave his code formation for formul{\ae} in place (taking appropriate measures to insure disjointness). Then, we obtain an $\ele$-adequate G\"odel numbering $\digamma$.  
		However, we still would have:
		$\digamma(\neg\, A ) := 2\cdot 3^{\,\digamma(A)}$. This means
		that, if we iterate negations, we run up an exponential tower.
		
		The example of $\digamma$ illustrates that it is possible that a 
		G\"odel numbering is $\ele$-adequate but not elementary in the length of
		expressions. This observation suggests that, where \emph{p-time} and \emph{primitive recursive} satisfy a certain equilibrium, possibly \emph{elementary} does not. If the tracking function for function application is exponential, the codes of numerals, both standard and efficient, grow too fast.
		So, the tracking function for function application is severely constrained. But if this function is so slow, it seems a certain imbalance to make other forms of application much faster. So, why is this numbering not already $\pol$-adequate? \dots 
	\end{remark}
	
	We now introduce a useful proof-theoretical characterisation of the classes $\pol$, $\ele$ and $\prc$. This will result in a specification of the theories in which $\mathcal C$-adequate numberings permit the arithmetisation of syntax. To this end, we first introduce so-called \textit{provably recursive} functions.
	
	\begin{definition}
		Let $T$ be an $\mathcal{L}$-theory. A function $f \colon \omega^k \to \omega$ is 
		called 
		\emph{$\Sigma_n$-definable in $T$}, if there exists a $\Sigma_n$-formula $A(x_1, \ldots, x_k,y)$ such that
		\begin{itemize}
			\item $\mathbb{N} \models A(\underline{m_1}, \ldots, \underline{m_k},\underline{p})$ iff $f(m_1, \ldots, m_k)=p$, for all $m_1, \ldots, m_k,p \in \omega$;
			\item $T \vdash \forall x_1, \ldots, x_k \exists ! y\, A(x_1, \ldots, x_k,y)$.
		\end{itemize}
		The function $f$ is called \emph{provably recursive in $T$} if $f$ is $\Sigma_1$-definable in $T$.
	\end{definition}
	
	\begin{remark}
		We note that there is some awkwardness to this definition since the provably recursive functions of all $\Sigma_1$-unsound theories extending {\sf EA}, like 
		$\mathrm{I}\Sigma_1+{\sf incon}({\sf ZF})$, are precisely
		all recursive functions. This creates the mistaken impression that, e.g., 
		$\mathrm{I}\Sigma_1+{\sf incon}({\sf ZF})$ is stronger than {\sf PA}.
		
		With an extra argument, one can show that this malaise persists even when we replace
		`$\mathbb N \models A(\underline{m_1}, \ldots, \underline{m_k},\underline{p})$' in the definition by `$T \vdash A(\underline{m_1}, \ldots, \underline{m_k},\underline{p})$'.
		
		We are unaware of attempts to address this awkwardness. 
	\end{remark}
	
	The classes $\ele$ and $\prc$ can now be characterised as exactly the functions which are provably recursive in a certain theory.
	The situation for the p-time computable functions is a bit more delicate. 
	
	\begin{theorem} Let $f$ be a number-theoretic function. Then
		\begin{itemize}
			\item $f$ is p-time iff $f$ is $\Sigma_1^{\sf b}$-definable in $\mathsf{S}^1_2$, where $\Sigma_1^{\sf b}$ is the special formula class introduced in \citep{buss:boun86}; 
			\item $f$ is elementary iff $f$ is provably recursive in $\mathsf{EA}$; 
			\item $f$ is primitive recursive iff $f$ is is provably recursive in $\mathsf{I \Sigma_1}$.
		\end{itemize}
	\end{theorem}
	
	Hence, for example, for any $\prc$-adequate numbering $\xi$ there exists a $\Sigma_1$-formula $\mathsf{Conj}(x,y,z) \in \mathcal{L}$, such that 
	\begin{itemize}
		\item $\mathbb{N} \models \mathsf{Conj}(\underline{\xi(A)},\underline{\xi(B)},\underline{\xi(C)}) $ iff $C = A \wedge B$, for all formul{\ae} $A,B,C \in \mathcal{L}$;
		\item $\mathsf{I \Sigma_1} \vdash \forall x,y \exists ! z \, \mathsf{Conj}(x,y,z) $;
	\end{itemize}
	i.e., the $\xi$-tracking function of $\dot\wedge$ is provably recursive in $\mathsf{I \Sigma_1}$.
	Indeed, the representation of $\dot\wedge$ by a provably recursive function (in $\mathsf{PA}$) is required as a necessary condition for reasonable numberings in \cite[p.~33]{Halbach2014}. As a result of the above theorem, $\prc$-adequate numberings allow a large portion of syntax to be formalised in the theory~$\mathsf{I \Sigma_1}$ and, similarly, for
	$\ele$ and {\sf EA} and for $\pol$ and ${\sf S}^1_2$.\footnote{We note that Halbach's constraint
		loses its meaning in $\Sigma_1$-unsound extensions of {\sf EA}. Also, from the technical point of view, the demand is probably too strong. For example, in \citep{haje:meta91}, we find
		an arithmetisation of syntax in $\mathrm I\Delta_0$. Here the substitution function is not provably total. However, we think that the Second Incompleteness Theorem can still be
		formalized. Clearly, these matters require a lot more attention than has been given
		until now.}
	
	Recall that the constructions of the numberings ${\sf gn}_0$, ${\sf gn}_1$ and ${\sf gn}_2$ rely on an enumeration $(A_n)_{n \in \omega}$ of $\mathcal{L}$-expressions. Thus far, assuming  $(A_n)_{n \in \omega}$ to be effective has sufficed to ensure the effectiveness of the resulting numberings ${\sf gn}_0$, ${\sf gn}_1$ and ${\sf gn}_2$. Since no further computational constraint has been imposed on $(A_n)_{n \in \omega}$, there is however no reason to expect these numberings to be even $\prc$-adequate. In the remainder of this paper, we will assume that the enumeration $(A_n)_{n \in \omega}$ (without repetitions) employed in the constructions of the numberings ${\sf gn}_0$, ${\sf gn}_1$ and ${\sf gn}_2$ is obtained from a standard numbering which is $\ele$-adequate, such 
	that $(A_n)_{n \in \omega}$ also satisfies the remaining assumptions imposed in the constructions of ${\sf gn}_0$, ${\sf gn}_1$ and ${\sf gn}_2$ respectively. That is, $(A_n)_{n \in \omega} = ( \mathsf{gn}_{\ast}^{-1}(n))_{n \in \omega}$, for some suitable standard numbering $\mathsf{gn}_{\ast}$.
	
	Under the assumption that $(A_n)_{n \in \omega}$ is $\ele$-adequate, we now show that the numberings ${\sf gn}_0$, ${\sf gn}_1$ and ${\sf gn}_2$ are $\ele$-adequate, and, thus, also $\prc$-adequate. This will follow from the fact that the standard numbering $\mathsf{gn}_{\ast}$ is $\ele$-adequate and that $\mathsf{gn}_{\ast}$, ${\sf gn}_0$, ${\sf gn}_1$ and ${\sf gn}_2$ are elementarily intertranslatable. Let $\mathsf{tr}_{i,j} \colon \mathsf{gn}_j(\mathcal{L}) \to \omega$ be the translation function, given by $\mathsf{tr}_{i,j} := \mathsf{gn}_i \circ \mathsf{gn}_j^{-1}$, for each $i, j \in \{ \ast,0,1,2 \}$. We then have:
	
	\begin{lem}
		\label{lem:translationsarepr}
		For each $i, j \in \{ \ast,0,1,2 \}$, the translation functions $\mathsf{tr}_{i,j}$ are elementary.
	\end{lem}
	
	\begin{proof}
		We treat the case of the intertranslation of $\ast$ and 2. To show that $\mathsf{tr}_{\ast,i}$ and $\mathsf{tr}_{i,\ast}$ are elementary for $i = 0,1$ proceeds similarly. From these cases we can conclude that $\mathsf{tr}_{i,j}$ is elementary for $i,j \leq 2$.
		
		We use the fact that the coding machinery for finite sequences can be developed in $\ele$ (see, for example, \citep[Chapter 2]{Schwichtenberg2012}). Let $\# \colon \omega^{< \omega} \to \omega$ be an elementary coding function of finite sequences and let ${\sf lh}$ be the length function. Let $g \colon \omega \to \omega$ serve as a parameter. We define the function $\sigma_g \colon \omega \to \omega$, by setting $\sigma_g(k) := \# \langle i_0, \ldots, i_{\ell-1} \rangle$, where $A_{i_0}, \ldots, A_{i_{\ell-1}}$ are exactly the sub-expressions of $A_k [{\sf c} := {\sf S}\widetilde{(k+2)}]$ which are not of the form $\widehat{m}$ such that  $g(n) \neq i_j$ for all $n < \nub(k-1)$ and $j < {\sf lh}(\sigma_g(k)) = \ell$. Since $\nub \in \ele$ and ${\sf gn}_\ast$ is $\ele$-adequate, $g \in \ele$ implies $\sigma \in \ele$, by the usual closure properties of elementary functions. In case that $g := \mathsf{tr}_{2,\ast}$, the last clause is equivalent to $A_{i_j} \notin \Lambda_{k-1}$ for each $j < \ell$.
		
		Set $\sigma := \sigma_{\mathsf{tr}_{2,\ast}}$. For any $n \in \omega$, we compute $\mathsf{tr}_{2,\ast}(n)$ by course-of-values recursion as follows. As usual, we will take the empty sum to be 0.
		
		\newcommand\textvtt[1]{{\normalfont\fontfamily{cmvtt}\selectfont #1}}
		
		\begin{quotation}
			\textvtt{Compute the smallest $k$ (${\leq} n)$ such that 
				$n\leq 2^{2k+4}$. Is $p := 2^{2k+4}-n < {\sf lh}(\sigma(k))$? 
				
				If \emph{yes}, set $\mathsf{tr}_{2,\ast}(n) := [\sigma(k)]_{{\sf lh}(\sigma(k)) -1 -p}$. 
				
				If \emph{no}, take
				\[
				m := n - \sum_{i=0}^{k-1}{\sf lh}(\sigma(i)).
				\]
				Set $\mathsf{tr}_{2,\ast}(n) := {\sf gn}_{\ast}(\widehat{m})$.} 
		\end{quotation}
		
		It is left to the reader to verify that this computation really defines the function~$\mathsf{tr}_{2,\ast}$. Note that when computing $\mathsf{tr}_{2,\ast}(n)$, we only resort to values of $\mathsf{tr}_{2,\ast}$ for arguments smaller than $n$, since $\nub(k-1) \leq n$. Moreover, let $\mathsf{B} \colon \omega \to \omega$ the elementary function given by $\mathsf{B}(k) := \mathsf{gn}_{\ast}(A_k [{\sf c} := {\sf S}\widetilde{(k+2)}])$. Note that $\mathsf{tr}_{2,\ast}(n)$ codes a filler expression of the form $\widehat{m}$, or a sub-expression of $A_k [{\sf c} := {\sf S}\widetilde{(k+2)}]$. Since in each stage $k$, there are less than $\nub(k)$-many fillers added, we obtain the estimate $\mathsf{tr}_{2,\ast}(n) \leq \max({\sf gn}_{\ast}(\widehat{(k+1) \cdot \nub(k)},\mathsf{B}(k))$. Hence, the employed recursion is \textit{limited}. Since elementary functions are closed under limited course-of-values recursion, we conclude that $\mathsf{tr}_{2,\ast} \in \ele$.
		
		In order to show that $\mathsf{tr}_{\ast,2}$ is elementary, we use the fact 
		that a function is elementary if (1) its graph is elementary and (2) it can be dominated by an elementary function. Since $\mathsf{tr}_{2,\ast}$ is elementary, it is easy to check that the graph of $\mathsf{tr}_{\ast,2}$ is elementary. Hence, it is sufficient to show (2). Let $A_k \in \mathcal{L}$. By Lemma~\ref{lem:AkinLambdak}, $A_k \in \Lambda_k$. Hence, $\mathsf{tr}_{\ast,2}(k) < \nub(k)$ and we are done.
	\end{proof}
	
	It is easy to show that the elementary (or primitive recursive) relations and functions are invariant regarding numberings which are elementarily (or primitive recursively) intertranslatable.
	
	\begin{lem}
		\label{lemma:prinvarianceofequivalentnumb}
		Let $\mathcal{C} \in \{\pol,\ele,\prc  \}$. Let $S_i$ be sets and let $\alpha_i, \beta_i \colon S_i \to \omega$ be injective functions such that the translation functions $\alpha_i \circ \beta_i^{-1}$ and $\beta_i \circ \alpha_i^{-1}$ are in $\mathcal{C}$, for $i \in \{1, \ldots ,k\}$ and $k \in \omega$. For every $R \subseteq S_1 \times \ldots \times S_k$, $Q \subseteq S_{k+1}$ and function $f \colon R \to Q$ we then have:
		\begin{itemize}
			\item The relation
			\[
			\vec\alpha (R) := \{ \langle \alpha_1(s_1), \ldots, \alpha_k(s_k) \rangle \mid \langle s_1,\ldots,s_k \rangle \in R \}
			\]
			is in $\mathcal{C}$ iff $\vec\beta (R)$ is in $\mathcal{C}$.
			\item The function $f_{\vec\alpha} \colon \vec\alpha(R) \to \omega$ given by
			\[
			f_{\vec\alpha}(m_1, \ldots, m_k) := \alpha_{k+1}f(\alpha_1^{-1}(m_1), \ldots, \alpha_k^{-1}(m_k))
			\]
			is in $\mathcal{C}$ iff the function $f_{\vec\beta} \colon \vec\beta(R) \to \omega$ is in $\mathcal{C}$. 
		\end{itemize}
	\end{lem}
	
	We have seen in this section that $\mathcal{PR}$-adequacy can be extracted as a necessary condition for reasonable numberings from the literature. Since $\mathsf{gn}_{\ast}$ is $\mathcal E$-adequate we conclude from Lemma~\ref{lem:translationsarepr} and Lemma~\ref{lemma:prinvarianceofequivalentnumb} that the numberings ${\sf gn}_0$, ${\sf gn}_1$ and ${\sf gn}_2$ are $\mathcal E$-adequate and, in particular, $\mathcal{PR}$-adequate. Thus, the computational constraint of $\mathcal{PR}$-adequacy (or $\mathcal E$-adequacy) is not sufficiently restrictive to deny these numberings the status of adequate formalisation choices.
	
	\begin{question}
		It remains open whether there are $\pol$-adequate versions of the constructions given in this paper. More specifically:
		\begin{itemize}
			\item Can we find $\pol$-adequate self-referential G\"odel numberings that are monotonic?
			\item Can we find $\pol$-adequate G\"odel numberings which satisfy the Strong Diagonal Lemma~\ref{lem:strongdiaglem} that are strongly monotonic?
		\end{itemize}
	\end{question}
	
	\section{Domination and Regularity}
	\label{section:regularity}
	
	We have seen in Section~\ref{section:numeralsasquotations} that strong monotonicity is not sufficient to rule out m-self-referential sentences formulated in $\mathcal{L}$. This is due to the fact that the \enquote{fixed point} term $t$ used in the proof of Lemma~\ref{lem:strongdiaglem} is not an efficient numeral. Hence, conditions M3 and M4($\overline{\cdot}$) do not force the value of $t$ to be smaller than the code of the formula containing $t$.
	
	For the remainder of this section, let $\mathcal{L}$, $\mathcal{K}$ be fixed languages with $\mathcal{L}^0 \subseteq \mathcal{K} \subseteq \mathcal{L}$, as specified in Section~\ref{subsection:language}. We introduce two related constraints on numberings of $\mathcal{L}$ which prohibit the code of a closed $\mathcal{K}$-term to be smaller than its value. We will see that together with requiring monotonicity, these constraints, applied to $\mathcal{L}$, rule out m-self-referential sentences formulated in~$\mathcal{L}$.
	
	\subsection{Domination}
	
	We start by defining the constraint of \textit{domination}.
	
	\begin{definition}
		Let $\xi$ be a numbering of $\mathcal{L}$. We say that $\xi$ is \textit{$\mathcal{K}$-dominating}, if it satisfies the following principle:
		\begin{enumerate}
			\item[M5($\mathcal{K}$).] For all $t \in \mathsf{ClTerm}_\mathcal{K}$ whose value is in $\xi(\mathcal{L})$: $\mathsf{ev}(t) \leq \xi( t )$.
		\end{enumerate}
		We say that $\xi$ is \textit{dominating}, if it is $\mathcal{L}$-dominating.
	\end{definition}
	
	Monotonic dominating numberings are strongly monotonic. More precisely, M5($\mathcal{K}$) implies M4$^\ast$($\nu$) 
	for every numeral function $\nu$ for $\mathcal K$ (cf.~Section~\ref{section:numeralsasquotations}):
	
	\begin{cor}
		Let $\nu$ be a numeral function for $\mathcal{K}$ and let $\xi$ be a numbering of $\mathcal{L}$. If $\xi$ satisfies M5$(\mathcal{K})$, then $\xi$ satisfies M4$^\ast(\nu)$.
	\end{cor}
	
	\begin{proof}
		We have $n = \mathsf{ev}(\nu(n)) \leq \xi(\nu(n))$ for all $n \in \omega$. 
		Hence, $\xi(A) \leq \xi( \nu(\xi(A)) )$ for all $A \in \mathcal{L}$.
	\end{proof}
	
	We now show that when employing monotonic and $\mathcal{K}$-dominating numberings, m-self-reference is not attainable in~$\mathcal{K}$. That is, taken together, these two constraints rule out the existence of strong fixed point terms in $\mathcal K$.
	
	\begin{lem}
		\label{lem:unattainabilityofSR}
		Let $\xi$ be a numbering of $\mathcal{L}$ satisfying M3 and M5(${\mathcal{K}})$. Then for all $\mathcal{L}$-formul{\ae} $A(x)$ and closed ${\mathcal{K}}$-terms $t$, $\mathbb{N} \not\models t = \underline{\xi(A(t))}$.
	\end{lem}	
	
	\begin{proof}
		Assume that there exists an $\mathcal{L}$-formula $A(x)$ and a closed $\mathcal{K}$-term $t$ 
		such that $\mathbb{N} \models t = \underline{\xi(A(t))}$. Using M3 and M5(${\mathcal{K}}$),
		we then derive the contradiction $\mathsf{ev}(t) = \xi(A(t)) > \xi(t) \geq \mathsf{ev}(t)$.
	\end{proof}
	
	Let $\xi$ be a standard numbering of $\mathcal{L}^+$, where $\mathcal{L}^+$ contains a term which represents the canonical strong diagonal function. That is, $\mathcal{L}^+$ represents the function which maps (the $\xi$-code of) an ${\mathcal{L}}^+$-formula $A(x)$ with $x$ free to (the $\xi$-code of) its diagonalisation $A(\underline{\xi(A)})$. As we have seen already, there exists an $\mathcal{L}^+$-formula $A(x)$ and a term $t \in {\mathcal{L}}^+$ such that $\mathbb{N} \models t = \underline{\xi(A[ x := t])}$ (Lemma~\ref{classicdiaglemma}). 
	
	The reason to define the constraints of domination with respect to a \emph{sub}-language of the domain of a numbering can be illustrated as follows. While $\xi$ cannot be dominating with respect to $\xi$'s domain (namely, $\mathcal{L}^+$) by Lemma~\ref{lem:unattainabilityofSR}, it is possibly dominating with respect to the sub-language $\mathcal{L}^0$ (for instance in case that $\xi$ is the length-first ordering). 
	
	However, as we will show in Section~\ref{subsection:nondominatingnumberings}, there are reasonable numberings which even fail to be $\mathcal{L}^0$-dominating. Moreover, there is a large class of standard numberings found in the literature which is not dominating with respect to languages which contain a function symbol for exponentiation or the smash function (see Section~\ref{subsubsection:smash}). These examples put considerable pressure on the view that domination is a necessary condition for reasonable choices of numberings. This constraint thus appears to be of no use to a philosophically adequate approach which aims to block the attainability of m-self-reference in arithmetic formulated in $\mathcal{L}^0$.
	
	Exactly the same remarks apply to the constraint of regularity, introduced in the next section.
	
	\subsection{Regularity}
	
	We now introduce another constraint on numberings, along the lines of a notion put forward in \cite[p.~17]{Heck2007}.\footnote{Heck's (\citeyear{Heck2007}) notion is an amalgam of our notion of regularity and monotonicity. In fact, Heck's notion of regularity is entailed by M2, M3 and M5(${\mathcal{L}^0}$) (with $\geq$ replaced by $>$).}
	
	\begin{definition}
		Let $\xi$ be a numbering of $\mathcal{L}$. We say that $\xi$ is $\mathcal{K}$-\textit{regular} if
		\begin{itemize}
			\item $\xi(c) \geq c^{\mathbb{N}}$ for all constant symbols $c \in \mathcal{K}$
			\item $\xi(f(t_1,\ldots,t_k)) > f^{\mathbb{N}}(\xi(t_1),\ldots,\xi(t_k))$ for all $k$-ary function symbols $f \in \mathcal{K}$
		\end{itemize}
		Moreover, we call $\xi$ \textit{regular}, if it is $\mathcal{L}$-regular.
	\end{definition}
	
	By induction we immediately get
	
	\begin{lem}
		\label{lem:codelargerthanvalue}
		Let $\xi$ be a $\mathcal{K}$-regular numbering of $\mathcal{L}$. Then $\xi$ satisfies M5$(\mathcal{K})$.
	\end{lem}
	
	Thus, $\mathcal{K}$-regular numberings are $\mathcal{K}$-dominating.
	
	\subsubsection{Example of a Regular Numbering}
	
	A large class of standard numberings found in the literature are ${\mathcal{L}^0}$-regular and hence in particular ${\mathcal{L}^0}$-dominating. As an example, we show that the length-first numbering of ${\mathcal{L}}$ is ${\mathcal{L}^0}$-regular, where $\mathcal L$ is given in Polish notation (see Section~\ref{subsection:language}). This follows as a corollary from a more general observation which has been suggested to us by Fedor Pakhomov.
	
	Let $S$ be a set of strings that is closed under the constructor operations $\dot{{\sf A}}$ and $\dot{{\sf M}}$, given by $\langle \alpha,\beta \rangle \mapsto {\sf A} \alpha \beta$ and $\langle \alpha,\beta \rangle \mapsto {\sf M} \alpha \beta$ respectively. 
	
	Let $\xi \colon S \to \omega$ be a numbering of $S$ and let $g \colon S^2 \to S$ be a binary operation on $S$. We call the function which maps $\xi$-codes $\langle m,n \rangle$ to $\xi(g(\xi^{-1}(m),\xi^{-1}(n)))$ the $\xi$-\emph{tracking function} of $g$. As usual, we call a function $f \colon \omega^2 \to \omega$ monotonic, if $f(x_1, x_2) \leq f(y_1, y_2)$ for all $x_1 \leq y_1$ and $x_2 \leq y_2$.
	
	\begin{lem}[Pakhomov]
		\label{lem:regularity}
		Let $\xi \colon S \to \omega$ be a bijection such that the $\xi$-tracking functions of $\dot{{\sf A}}$ and $\dot{{\sf M}}$ are monotonic. We have for all $\alpha, \beta \in S$:
		\[
		\xi({\sf A} \alpha \beta), \xi({\sf M} \alpha \beta) > \xi(\alpha) \cdot \xi(\beta).
		\]
	\end{lem}
	
	\begin{proof}
		We show the claim for ${\sf M} \alpha \beta$. Let $\alpha, \beta \in S$. Since $\xi$ is bijective, there are $(\xi(\alpha)+1) \cdot (\xi(\beta)+1)$-many pairs of expressions $\langle \alpha',\beta' \rangle$ such that $\xi(\alpha') \leq \xi(\alpha)$ and $\xi(\beta') \leq \xi(\beta)$. For each such pair $\langle \alpha',\beta' \rangle$, we have $\xi({\sf M} \alpha' \beta') \leq \xi({\sf M} \alpha \beta)$, since the tracking function of $\dot{{\sf M}}$ is monotonic. Moreover, for any $\langle \alpha',\beta' \rangle \neq \langle \alpha'',\beta'' \rangle$ we have $\xi({\sf M} \alpha' \beta') \neq \xi({\sf M} \alpha'' \beta'')$ by the injectivity of $\dot{{\sf M}}$ and $\xi$. We thus conclude
		\[
		\xi(\alpha) \cdot \xi(\beta) + \xi(\alpha) +\xi(\beta) = (\xi(\alpha)+1) \cdot (\xi(\beta)+1)-1 \leq \xi({\sf M} \alpha \beta).
		\]
		In particular, $\xi(\alpha) \cdot \xi(\beta) < \xi({\sf M} \alpha \beta)$, if one of $\xi(\alpha)$, $\xi(\beta)$ is $ > 0$. Moreover, if both of $\xi(\alpha)$, $\xi(\beta)$ are $0$, then, by injectivity, ${\sf M}\alpha\beta$ cannot be 0, so we have 
		$\xi({\sf M}\alpha\beta) >0 =
		\xi(\alpha)\cdot\xi(\beta)$.
	\end{proof}
	
	Let $\mathcal L$ be given in Polish notation and let $\mathcal A$ be the alphabet of $\mathcal L$. Let $\mathfrak g$ be the length-first numbering of $\mathcal{A}^\ast$. Note that since $\mathfrak g$ is bijective and monotonic with respect to the sub-string relation, the assumptions of Lemma~\ref{lem:regularity} are satisfied (for $S := \mathcal{A}^\ast$). It is then easy to derive the $\mathcal{L}^0$-regularity of $\mathfrak g$.
	
	\begin{cor}
		\label{cor:lengthfirstisregular}
		The length-first numbering $\mathfrak g$ is $\mathcal{L}^0$-regular.
	\end{cor}
	
	For any infinite subset $Y$ of $\omega$ we define ${\sf coll}_Y$ to be the unique
	order preserving bijection from $Y$ to $\omega$.
	
	\begin{remark}
		\label{remark:collapsednumbering}
		Let $X := \mathfrak g(\mathcal L)$.
		Consider the numbering $\mathfrak h := {\sf coll}_X \circ \mathfrak g$ of $\mathcal{L}$. 
		The numbering $\mathfrak h$ can be seen as a length-first enumeration of the well-formed expressions in $\mathcal{L}$ (as opposed to considering arbitrary strings of $\mathcal A$). Arguably, $\mathfrak h$ is a reasonable choice of a G\"odel numbering. 
		
		Is $\mathfrak h$ $\mathcal{L}^0$-regular? Lemma~\ref{lem:regularity} cannot be applied to $\mathfrak h$, since its domain $\mathcal{L}$ is not single-sorted and thus not closed under the constructor operations $\dot{{\sf A}}$ and $\dot{{\sf M}}$. 
		However, we can use this lemma to show that the length-first enumeration of the closed-term fragment of $\mathcal{L}^0$ is $\mathcal{L}^0$-regular. Let $\mathcal{L}^{\circ}$ be given by
		\begin{itemize}
			\item
			$t ::= \zero \mid {\sf S}t \mid {\sf A} t t \mid {\sf M} t t$
		\end{itemize}
		Let $X^\circ := \mathfrak g(\mathcal L^\circ)$. We take 
		$\mathfrak h^{\circ} := {\sf coll}_{X^\circ} \circ \mathfrak g$ to be the numbering of $\mathcal{L}^{\circ}$. Note $\mathcal{L}^{\circ}$ is closed under both $\dot{{\sf A}}$ and $\dot{{\sf M}}$. Hence, we can conclude from Lemma~\ref{lem:regularity} that $\mathfrak h^{\circ}$ is $\mathcal{L}^0$-regular. 
	\end{remark}
	
	\subsubsection{Example of a Non-Regular Numbering}
	\label{subsection:nondominatingnumberings}
	
	We define a non-regular numbering of $\mathcal{L}^0$, given in Polish notation (see Section~\ref{subsection:language}).\footnote{We are indebted to the MathOverflow users Matt F., Fedor Pakhomov and Konrad Zdanowski for their helpful suggestions which inspired our construction of the numbering ${\sf gn}_3$ (see \citep{Zdanowski2020}).} Let $\mathcal{L}^{\star}$ denote the following extension of $\mathcal{L}^0$:
	
	\begin{itemize}
		\item
		$x::= {\sf v} \mid x'$
		\item
		$t ::= \zero \mid x \mid  {\sf S}t \mid {\sf A}tt \mid {\sf A}d \mid {\sf M}tt \mid {\sf M} d$
		\item
		$A ::= \bot \mid \top \mid {=}tt \mid {=}d \mid \neg A \mid \wedge A A \mid \wedge D \mid \vee A 
		A \mid \vee D \mid {\to} A A \mid {\to} D \mid \forall x\,A \mid \exists x\, A$
		\item
		$d ::= \delta t$
		\item
		$D ::= \Delta A$
	\end{itemize}
	
	The language $\mathcal{L}^{\star}$ can be seen to implement the sharing of certain sub-expressions.
	That is, $\mathcal{L}^0$-expressions of the form $Bt t$, for any binary constructor symbol $B$,
	are \enquote{contracted} to expressions of the form $B \delta t$. This contraction and its reversal is performed
	by the translation functions $\mathfrak c \colon \mathcal{L}^0 \to \mathcal{L^\star}$ and
	$\mathfrak t \colon \mathcal{L^\star} \to \mathcal{L}^0$, defined recursively as follows:
	Both translations commute with all construction steps for $t$ and $A$ involving the zero-ary operations, the unary operations
	and the quantifiers. Moreover:
	
	\begin{multicols}{2}
		\begin{itemize}
			%\item $\mathfrak c(\mathsf{v}) := \mathsf{v}$;
			%\item $\mathfrak c(x') := \mathfrak c(x)'$;
			%\item $\mathfrak c(\mathsf{0}) := \mathsf{0}$;
			%\item $\mathfrak c(\bot) := \bot$;
			%\item $\mathfrak c(\top) := \top$;
			%\item $\mathfrak c({\sf U} t) := {\sf U} \mathfrak c(t)$;
			\item
			$\mathfrak c({\sf A} t u) :=
			\begin{cases}
			{\sf A} \delta \mathfrak c(t) & \text{if } t = u;\\
			{\sf A} \mathfrak c(t) \mathfrak c(u) & \text{if } t \neq u;\\
			\end{cases}$
			similarly, for  {\sf M};
			\item
			$\mathfrak c({=} t u) :=
			\begin{cases}
			{=} \delta \mathfrak c(t) & \text{if } t = u;\\
			{=} \mathfrak c(t) \mathfrak c(u) & \text{if } t \neq u;\\
			\end{cases}$
			\item
			$\mathfrak c(\wedge AB) :=
			\begin{cases}
			\wedge \Delta \mathfrak c(A) & \text{if } A = B;\\
			\wedge \mathfrak c(A) \mathfrak c(B) & \text{if } A \neq B;\\
			\end{cases}$	\\
			similarly, for $\vee$ and $\to$.
		\end{itemize}
		\columnbreak
		\begin{itemize}
			%\item $\mathfrak t(\mathsf{v}) := \mathsf{v}$;
			%\item $\mathfrak t(x') := \mathfrak t(x)'$;
			%\item $\mathfrak t(\mathsf{0}) := \mathsf{0}$;
			%\item $\mathfrak t(\bot) := \bot$;
			%\item $\mathfrak t(\top) := \top$;
			%\item $\mathfrak t({\sf V} t) := {\sf V} \mathfrak t(t)$;
			\item $\mathfrak t(\mathsf{A} \delta t) := \mathsf{A} \mathfrak t(t) \mathfrak t (t)$;\\
			similarly, for {\sf M};
			\item $\mathfrak t({=} \delta t) := {=} \mathfrak t(t) \mathfrak t (t)$;
			\item $\mathfrak t(\wedge\Delta A) := \wedge \mathfrak t(A) \mathfrak t(A)$;\\
			similarly for $\vee$ and $\to$.
		\end{itemize}
	\end{multicols}
	
	%where ${\sf U} \in \{ \mathsf{S},  \neg \}$, ${\sf B} \in \{ \mathsf{A}, \mathsf{M},
	%=, \wedge, \vee, \rightarrow, \forall, \exists \}$ and ${\sf V} \in \{ \mathsf{S},
	%\neg, \mathsf{A}, \mathsf{M}, =, \wedge, \vee, \rightarrow \}$.
	\noindent
	The translation $\mathfrak t$ is well-defined. This follows from unique readability, which is proved in the lemma below.
	
	\begin{lem}
		$\mathcal{L}^\star$ satisfies unique readability.
	\end{lem}
	
	\begin{proof}[First Proof.]
		Consider the language $\hat{\mathcal L}$ that consists of the arithmetical language extended with
		a new constant $\delta$ and a new (zero-ary) proposition symbol $\Delta$.
		We prove by an easy induction that every syntactical string of $\mathcal L^\star$
		is a syntactical string of $\hat{\mathcal L}$. Now suppose an $\mathcal L^\star$ syntactical object has two readings.
		We consider e.g. the case where one of those readings is ${\sf M}\delta t$. Clearly, the other reading can only be of the form
		${\sf M}\delta t$ since no $\mathcal L^\star$-term can begin with $\delta$. 
		Now consider the case where one reading is ${\sf M}t_1t_2$. The other reading must be of the form ${\sf M}u_1u_2$.
		We apply the embedding into $\hat{\mathcal L}$ and unique reading in $\hat{\mathcal L}$, to see that $t_1=u_1$ and $t_2=u_2$.
	\end{proof}
	
	\begin{proof}[Second Proof.]
		Consider the language $\breve{\mathcal L}$ that consists of the arithmetical language extended with new
		unary functions ${\sf A}_\delta$ and ${\sf M}_\delta$, a new unary predicate symbol $=_\delta$ and new unary
		connectives $\wedge_{\Delta}$, $\vee_\Delta$ and $\to_\Delta$. We map $\mathcal L^\star$ into $\breve{\mathcal L}$
		by the obvious mapping $\mathfrak d$ that sends, e.g., ${\sf A}t_1t_2$ to ${\sf A}\mathfrak d(t_1) \mathfrak d(t_2)$ and
		${\sf A}\delta t $ to ${\sf A}_\delta \mathfrak d(t)$.
		It is easy to see that $\mathfrak d$ is well-defined (since no $\mathcal L^\star$-term can start with
		$\delta$ and no $\mathcal L^\star$-formula can start with $\Delta$) and injective.
		
		Now suppose an $\mathcal L^\star$ syntactical object has two readings.
		We consider, e.g., the case where one of those readings is ${\sf M}\delta t$. Clearly, the other reading can only be of the form
		${\sf M}\delta t$ since no $\mathcal L^\star$-term can begin with $\delta$.
		Now consider the case where one reading is ${\sf M}t_1t_2$. The other reading must be of the form ${\sf M}u_1u_2$.
		We apply the embedding $\mathfrak d$ into $\hat{\mathcal L}$ and unique reading in $\hat{\mathcal L}$, to see that $\mathfrak d(t_1)=\mathfrak d(u_1)$ 
		and $\mathfrak d(t_2)=\mathfrak d(u_2)$.  By the injectivity of $\mathfrak d$, we find $t_1=u_1$ and $t_2=u_2$.
	\end{proof}
	
	Let $\mathcal A$ be the alphabet of $\mathcal L^{\star}$ with some fixed ordering. Let $\mathfrak{g}$ be the length-first numbering of $\mathcal A^\ast$.
	We define our numbering ${\sf gn}_3$ of $\mathcal{L}^0$ by setting ${\sf gn}_3(\phi) := \mathfrak{g}(\mathfrak c(\phi))$.
	Clearly, ${\sf gn}_3$ is injective, effective and (strongly) monotonic. 
	We show that ${\sf gn}_3$ is not regular. 
	
	\begin{theorem}
		${\sf gn}_3$ is not regular. 
	\end{theorem}
	
	\begin{proof}
		Since the size of $\mathcal A$ is 17, by Lemma~\ref{lem:lengthfirstestimates} we have for each $\alpha \in \mathcal A^\ast$
		\[
		\frac{17^{|\alpha|}-1}{16} \leq \mathfrak{g}(\alpha)  <  \frac{17^{|\alpha|+1}-1}{16}.
		\]
		Hence,
		\begin{itemize}
			\item ${\sf gn}_3(\underline{n}) = \mathfrak{g}(\underline{n}) \geq \frac{17^{n+1}-1}{16}$;
			\item ${\sf gn}_3(\mathsf{M}\underline{n} \, \underline{n}) = \mathfrak{g}(\mathsf{M} \delta \underline{n}) < \frac{17^{n+4}-1}{16}$.
		\end{itemize}
		But then, for all $n > 4$,\qedright
		\begin{align*}
		{\sf gn}_3(\underline{n}) \cdot {\sf gn}_3(\underline{n}) & \geq \left(\frac{17^{n+1}-1}{16}\right)^{2}\\
		& \geq \left(\frac{17^{n}}{16}\right)^{2}\\
		& \geq \frac{17^{2n-1}}{16}\\
		& \geq \frac{17^{n+4}-1}{16}\\
		& > {\sf gn}_3(\mathsf{M}\underline{n} \, \underline{n}).
		\end{align*}
	\end{proof}
	
	\noindent
	We note that the computation in the above proof
	does not depend on the specific size of the alphabet, here represented by the numbers $17$ and $16=17-1$.
	It works as long as the alphabet has at least two symbols. So, certainly, it works for all extensions for $\mathcal{L}^0$.
	
	\begin{remark}
		The present example of a non-regular G\"odel numbering touches lightly upon a theme that we do not discuss in this paper: G\"odel numberings based on alternative notions of
		syntax. Our treatment is mainly inspired by string approaches and algebraic approaches. In our example we look at sharing in the narrowest
		possible sense, but
		there are more principled and more encompassing approaches to sharing involving directed acyclic graphs and the like.
		
		Similarly, there are different approaches to variables like de Bruijn notations and the Peirce-Quine linking notation.
		In Peirce diagrams and in the clause-set notation, we abstract away from the order of finite conjunctions. Etcetera.
	\end{remark}
	
	We think that ${\sf gn}_3$ is a perfectly reasonable G\"odel numbering.
	We strengthen this intuition by showing that ${\sf gn}_3$ has good properties. 
	
	\begin{theorem}
		${\sf gn}_3$ is monotonic. Moreover, it is strongly monotonic w.r.t. standard and efficient numerals.
	\end{theorem}
	
	\begin{proof}
		We note that $\mathfrak c$ preserves direct sub-expressions, in other words, if $\tau$ is a direct sub-expression of $\rho$, then
		$\mathfrak{c}(\tau)$ is a direct sub-expression of $\mathfrak c(\rho)$. It follows that $\mathfrak c$ preserves sub-expressions.
		Since $\mathfrak g$ is monotonic, we may conclude that ${\sf gn}_3 = \mathfrak g \circ \mathfrak c$ is also monotonic.
		Moreover, it is easy to see that $\mathfrak c$ transforms standard numerals and efficient numerals identically. Hence, 
		${\sf gn}_3$ inherits strong monotonicity both with respect to standard numerals and to efficient numerals from $\mathfrak g$.
	\end{proof}
	
	\begin{remark}
		We can easily adapt the numerals introduced in Section~\ref{prideco} to the present situation in order to obtain
		numerals for which strong monotonicity fails w.r.t.\ ${\sf gn}_3$. The numeral function on $\mathcal{L}^0$ would look as
		follows.
		
		\begin{enumerate}[i.]
			\item ${\sf pd}^\star(0) := \zero$ and ${\sf pd}^\star(1) := {\sf S}\zero$.
			\item Suppose $n > 1$, and $n$ is not a power of a prime. Let $p$ be the smallest prime
			that divides $n$. Say $n = p^i\cdot m$, were $p$ does not divide $m$ and $i>0$. Then ${\sf pd}^\star(n) := \times {\sf pd}^\star(p^i){\sf pd}^\star(m)$.
			\item Suppose $n = p^i$, where $p$ is prime and $i>0$. We take  ${\sf pd}^\star(p) := \underline p$. In case $i = 2k$, for $k>0$,
			we set ${\sf pd}^\star(n) := \times {\sf pd}^\star(p^k){\sf pd}^\star(p^k)$. In case $i = 2k+1$, for $k>0$, we set 
			${\sf pd}^\star(n) := \times \underline p {\times} {\sf pd}^\star(p^k){\sf pd}^\star(p^k)$.
		\end{enumerate}
		It is easy to see that ${\sf pd}^\circ := \mathfrak c \circ {\sf pd}^\star$ behaves like {\sf pd} from Example~\ref{prideco}
		with ${\sf M}\delta$ in the role of {\sf Q}. It follows that ${\sf gn}_3(2^{2^n})$ grows exponentially in $n$.
	\end{remark}
	
	\subsubsection{Smash and Exponentiation}
	\label{subsubsection:smash}
	
	Let $\#$ be a binary function symbol with Nelson's (\citeyear{Nelson1986}) smash function, given by $(m,n) \mapsto 2^{|m| \cdot |n|}$, as its intended interpretation. Let ${\sf exp}$ be a unary function symbol for exponentiation. We set $\mathcal{L^{\#}} := \mathcal{L}^0 \cup \{ \# \}$ and $\mathcal{L^{\sf exp}} := \mathcal{L}^0 \cup \{ {\sf exp} \}$.
	We now show that there are standard numberings found in the literature which are neither $\mathcal{L^{\#}}$-regular nor ${\mathcal{L}}^{\mathsf{exp}}$-regular.
	
	Let $\xi$ be a numbering such that there exists a polynomial $P$ with
	\begin{equation}
	\tag{$\ast$}
	\label{equ:numberingorder}
	\xi(A) < 2^{P(|A|)} \text{ for all } A.
	\end{equation}
	For instance, \citep{buss:boun86} and \citep{buss:hand98} contain numberings of $\mathcal{L^{\#}}$ with this property. Moreover, length-first numberings, such as our numbering $\mathfrak g$, satisfy (\ref{equ:numberingorder}). To see this, let $\mathfrak g$ be a length-first numbering over an alphabet with size $N > 1$. 
	Consider any $m \in \omega$  such that $N \leq 2^m$. By Lemma~\ref{lem:lengthfirstestimates} we find, for all $A$:
	\[
	\mathfrak g(A) < \frac{N^{|A|+1}-1}{N-1} < N^{|A|+1} \leq  (2^m)^{|A|+1} = 2^{m \cdot |A|+m}.
	\]
	
	We now show that any numbering $\xi$ satisfying ($\ast$) is neither $\mathcal{L^{\#}}$-regular, nor $\mathcal{L^{\mathsf{exp}}}$-regular.
	
	\begin{theorem}
		$\xi$ is not $\mathcal{L^{\#}}$-regular.
	\end{theorem}
	
	\begin{proof}
		We define:
		\begin{itemize}
			\item
			$\widehat 1 := ({\sf SS}\zero \mathbin{\#} {\sf SS}\zero)$
			\item
			$\widehat{n+1} := ({\sf SS}\zero \mathbin{\#} \widehat{n})$.
		\end{itemize}
		
		It is easy to check that ${\sf ev}(\widehat n ) = 2^{2^{n+1}}$ and $|\widehat n| = 6n+3$.  Since $P$ is a polynomial there exists $n \in \omega$ such that $P(6n+3) < 2^{n+1}$ and hence
		\[
		\xi(\widehat{n}) < 2^{P(6n+3)} < 2^{2^{n+1}} = {\sf ev}(\widehat n ).
		\]
		Thus by, Lemma~\ref{lem:codelargerthanvalue}, $\xi$ is not $\mathcal{L^{\#}}$-regular.
	\end{proof}
	
	\begin{theorem}
		$\xi$ is not $\mathcal{L^{\mathsf{exp}}}$-regular.
	\end{theorem}
	
	\begin{proof}
		We define:
		\begin{itemize}
			\item
			$\wideparen{0} := {\sf SS}\zero$
			\item
			$\wideparen{n+1} := \mathsf{exp}(\wideparen{n})$.
		\end{itemize}
		
		We then have ${\sf ev}(\wideparen n ) = 2 \upuparrows (n+1)$ and $|\wideparen n| = 3n+3$.\footnote{We use here Knuth's up-arrow notation.} Since $P$ is a polynomial there exists $n \in \omega$ such that
		\[
		\xi(\wideparen{n}) < 2^{P(3n+3)} < 2 \upuparrows (n+1) = {\sf ev}(\wideparen n ).
		\]
		Thus, by Lemma~\ref{lem:codelargerthanvalue}, $\xi$ is not $\mathcal{L^{\mathsf{exp}}}$-regular.
	\end{proof} 
	
	\section{Type-free Truth Theories}
	\label{section:discussion}
	
	Building on work by \cite{Heck2007} and \cite{Schindler2015}, we now apply the results obtained in this paper to the study of type-free axiomatic truth theories. More specifically, we show that the consistency of certain type-free truth theories depends on the employed formalisation choices. Consider for instance the following two important principles of truth:\footnote{For the philosophical significance of these principles see \citep{Heck2004}.}
	\begin{description}
		\item[NOT] The negation of a sentence $A$ is true iff $A$ is not true;
		\item[DISQ] A sentence of the form $\gnum{t \text{ is true}}$ is true iff $t$ denotes a true sentence.
	\end{description}
	
	\noindent
	Taken together, these principles can intuitively be shown to be inconsistent. In order to show this, consider the Liar sentence\\
	$(L)$: $L$ is not true.
	
	Since $L$ is the sentence \enquote{$L$ is not true}, $L$ is true iff \enquote{$L$ is not true} is true. By NOT, \enquote{$L$ is not true} is true iff
	\enquote{$L$ is true} is not true. By DISQ, \enquote{$L$ is true} is not true iff, $L$ is not true. Hence, $L$ is true iff $L$ is not true. Contradiction (see \cite[p.12]{Heck2007}).
	
		Using a similar argument, we can intuitively also show the inconsistency of the following truth principle:	
		\begin{description}
			\item[NDISQ] A sentence of the form $\gnum{t \text{ is not true}}$ is true iff $t$ denotes a sentence which is not true.
		\end{description}
		\noindent
		Following \cite{Heck2007} and \cite{Schindler2015} we ask in which arithmetical framework this informal reasoning can be captured. Let $\mathcal{L}$ be any primitive recursive extension of ${\mathcal{L}^0}$, i.e., the result of adding function symbols for certain primitive recursive functions to ${\mathcal{L}^0}$. We set $\mathcal{L}_{\mathsf{T}} := \mathcal{L} \cup \{ \mathsf{T} \}$, where $\mathsf{T}$ is a new unary predicate symbol. The theory $\mathsf{R}_{{\mathcal{L}}}$ serves as a (weak) syntax theory in the formalisation of informal truth principles. Let moreover $\nu$ be a numeral function for $\mathcal{L}$ and let $\xi$ be a numbering of $\mathcal{L}_{\mathsf{T}}$. We now consider the following arithmetical formalisations of (NOT), (DISQ) and (NDISQ):
		\begin{description}
			\item[Not$(\nu,\xi)$] $\mathsf{T}(\nu(\xi(\neg A))) \leftrightarrow \neg \mathsf{T}(\nu(\xi(A)))$;
			\item[Disq$(\nu,\xi)$] $\mathsf{T}(\nu(\xi(\mathsf{T}(t)))) \leftrightarrow \mathsf{T}(t)$;
			\item[Disq$^\ast(\nu,\xi)$] $ \mathsf{T}(\nu(\xi(\mathsf{T}(\nu(\xi(A)))))) \leftrightarrow \mathsf{T}(\nu(\xi(A)))$;
			\item[NDisq$(\nu,\xi)$] $\mathsf{T}(\nu(\xi(\neg \mathsf{T}(t)))) \leftrightarrow \neg \mathsf{T}(t)$;
			\item[NDisq$^\ast(\nu,\xi)$] $\mathsf{T}(\nu(\xi(\neg \mathsf{T}(\nu(\xi(A)))))) \leftrightarrow \neg \mathsf{T}(\nu(\xi(A)))$.
		\end{description}
		where $A$ is any $\mathcal{L}_{\mathsf{T}}$-sentence and $t$ is any closed $\mathcal{L}$-term. Roughly speaking, Disq$(\nu,\xi)$ formalises (DISQ) by using all singular terms (of the given language) as names of sentences, while the formalisation Disq$^\ast(\nu,\xi)$ only employs certain canonical names, namely, $\nu$-numerals. The same holds for NDisq$(\nu,\xi)$ and NDisq$^\ast(\nu,\xi)$.  For every formalisation choice $\langle \mathcal{L}_{\mathsf{T}}, \nu, \xi \rangle$ we consider the formal truth theories	
		\begin{align*}
		\mathcal{S}(\mathcal{L}_{\mathsf{T}},\nu,\xi) & := \mathsf{R}_{{\mathcal{L}}_{\mathsf{T}}} + \text{Not}(\nu,\xi) + \text{Disq}(\nu,\xi);\\
		\mathcal{S}^\ast(\mathcal{L}_{\mathsf{T}},\nu,\xi) & := \mathsf{R}_{{\mathcal{L}}_{\mathsf{T}}} + \text{Not}(\nu,\xi) + \text{Disq}^\ast(\nu,\xi);\\
		\mathcal{T}(\mathcal{L}_{\mathsf{T}},\nu,\xi) & := \mathsf{R}_{{\mathcal{L}}_{\mathsf{T}}} + \text{NDisq}(\nu,\xi);\\
		\mathcal{T}^\ast(\mathcal{L}_{\mathsf{T}},\nu,\xi) & := \mathsf{R}_{{\mathcal{L}}_{\mathsf{T}}} + \text{NDisq}^\ast(\nu,\xi).
		\end{align*}
		
		\noindent
		In general, different formalisation choices $\langle \mathcal{L}_{\mathsf{T}}, \nu, \xi \rangle$ yield different theories.
		
		We now address the question under which constraints on the underlying formalisation choices $\langle \mathcal{L}_{\mathsf{T}}, \nu, \xi \rangle$ the $\mathcal{L}_{\mathsf{T}}$-theories $\mathcal{S}(\mathcal{L}_{\mathsf{T}},\nu,\xi)$, $\mathcal{S}^\ast(\mathcal{L}_{\mathsf{T}},\nu,\xi)$, $\mathcal{T}(\mathcal{L}_{\mathsf{T}},\nu,\xi)$ and $\mathcal{T}^\ast(\mathcal{L}_{\mathsf{T}},\nu,\xi)$ are inconsistent. That is, under which conditions do these theories provide a faithful formalisation of the intuitive reasoning regarding (NOT) \& (DISQ) as well as (NDISQ)? 
		
		The reader is reminded that ${\mathcal{L}}^+$ is a language which results from adding function symbols for primitive recursive functions to $\mathcal{L}^0$, such that ${\mathcal{L}}^+$ contains a term which represents the canonical strong diagonal function. Building on work by Heck and Schindler, the following answer can be extracted from the results of this paper. 
		
		\begin{theorem}
			Let $\nu$ be given such that $\nu = \underline{\cdot}$ or $\nu = \overline{\cdot}$. Let $\gamma$ be an $\mathcal E$-adequate and strongly monotonic numbering of ${\mathcal{L}}^+_{\mathsf{T}}$ for efficient numerals, which is $\mathcal{L}^0$-regular. We then have
			\begin{multicols}{2}
				\begin{enumerate}[1.]
					\item $\mathcal{S}({\mathcal{L}^0_{\mathsf{T}}},\nu,{\sf gn}_2) \vdash \bot$
					\item $\mathcal{S}^\ast({\mathcal{L}}^+_{\mathsf{T}},\overline{\cdot},{\sf gn}_0) \vdash \bot$
					\item $\mathcal{S}({\mathcal{L}^0_{\mathsf{T}}},\nu,\gamma) \not\vdash \bot$
					\item $\mathcal{S}^\ast({\mathcal{L}}^+_{\mathsf{T}},\nu,\gamma) \not\vdash \bot$
				\end{enumerate}
				\columnbreak
				\begin{enumerate}[1.]\setcounter{enumi}{4}
					\item $\mathcal{T}({\mathcal{L}^0_{\mathsf{T}}},\nu,{\sf gn}_2) \vdash \bot$
					\item $\mathcal{T}^\ast({\mathcal{L}}^+_{\mathsf{T}},\overline{\cdot},{\sf gn}_0) \vdash \bot$
					\item $\mathcal{T}({\mathcal{L}^0_{\mathsf{T}}},\nu,\gamma) \not\vdash \bot$
					\item $\mathcal{T}^\ast({\mathcal{L}}^+_{\mathsf{T}},\nu,\gamma) \not\vdash \bot$
				\end{enumerate}
			\end{multicols}
		\end{theorem} 
		
		\begin{proof}
			(3) follows from \cite[Theorem 1]{Heck2007} and (7) is Proposition 2.2.6 in \citep{Schindler2015}. (4) and (8) can be shown similarly, using the observation that $\nu$-numerals are $\mathcal{L}^0$-terms. See also \cite[p.~13]{Heck2007}. 
			
			We prove (2). By Lemma~\ref{lem:gn3isselfref} there exists $n \in \omega$ such that (Eq) $n = {\sf gn}_0(\neg \mathsf{T}(\overline{n}))$. We can then derive in $\mathcal{S}^\ast({\mathcal{L}}^+_{\mathsf{T}},\overline{\cdot},{\sf gn}_0)$ the following contradiction
			\begin{align*}
			\mathsf{T}(\overline{n}) & \leftrightarrow \mathsf{T}(\overline{{\sf gn}_0(\neg \mathsf{T}(\overline{n}))}), \text{ using (Eq)}\\
			& \leftrightarrow \neg \mathsf{T}(\overline{{\sf gn}_0(\mathsf{T}(\overline{n}))}), \text{ using Not$(\overline{\cdot},{\sf gn}_0)$}\\
			& \leftrightarrow \neg \mathsf{T}(\overline{{\sf gn}_0(\mathsf{T}(\overline{{\sf gn}_0(\neg \mathsf{T}(\overline{n}))}))}), \text{ using (Eq)}\\
			& \leftrightarrow \neg \mathsf{T}(\overline{{\sf gn}_0(\neg \mathsf{T}(\overline{n}))}), \text{ using Disq$^\ast(\overline{\cdot},{\sf gn}_0)$}\\
			& \leftrightarrow \neg \mathsf{T}(\overline{n}), \text{ using (Eq)}
			\end{align*}
			The proof of (6) proceeds similarly. (1) and (5) can be shown by applying Lemma~\ref{lem:strongdiaglem}.
		\end{proof}
		
		\noindent
		The results of the theorem can be summarised in the following table:
		
		\begin{center}
			\begin{tabular}{ |l|l|l| } 
				\hline
				Constraints on numberings & consistent & inconsistent\\
				\hline \Tstrut
				$\mathcal E$-adequate \& strongly~monotonic & $\mathcal{S}({\mathcal{L}^0_{\mathsf{T}}},\nu,\gamma)$  & $\mathcal{S}({\mathcal{L}^0_{\mathsf{T}}},\nu,{\sf gn}_2)$  \\
				& $\mathcal{T}({\mathcal{L}^0_{\mathsf{T}}},\nu,\gamma)$ &  $\mathcal{T}({\mathcal{L}^0_{\mathsf{T}}},\nu,{\sf gn}_2)$\Tstrut\\
				\hline \Tstrut
				$\mathcal E$-adequate \& monotonic & $\mathcal{S}^\ast({\mathcal{L}}^+_{\mathsf{T}},\nu,\gamma)$ & $\mathcal{S}^\ast({\mathcal{L}}^+_{\mathsf{T}},\overline{\cdot},{\sf gn}_0)$ \\
				& $\mathcal{T}^\ast({\mathcal{L}}^+_{\mathsf{T}},\nu,\gamma)$ & $\mathcal{T}^\ast({\mathcal{L}}^+_{\mathsf{T}},\overline{\cdot},{\sf gn}_0)$\Tstrut \\
				\hline
			\end{tabular}
		\end{center}
		\
		\\
		\noindent
		Thus, the conditions of $\mathcal E$-adequacy and monotonicity taken together are not restrictive enough to determine the consistency of $\mathcal{S}^\ast({\mathcal{L}}^+_{\mathsf{T}},\overline{\cdot},\xi)$ and of $\mathcal{T}^\ast({\mathcal{L}}^+_{\mathsf{T}},\overline{\cdot},\xi)$. Moreover, the conditions of $\mathcal E$-adequacy and strong monotonicity for efficient numerals are together not restrictive enough to determine the consistency of $\mathcal{S}(\mathcal{L}^0_{\mathsf{T}},\nu,\xi)$ and of $\mathcal{T}(\mathcal{L}^0_{\mathsf{T}},\nu,\xi)$.

	\section{Summary and Perspectives}
	
	In our paper we have studied various requirements on formalisation choices, such as the G\"odel numbering, under which m-self-reference is attainable in arithmetic. In particular, we have examined the admissibility of self-referential numberings, which provide immediate means to formalise m-self-reference. We have shown that monotonicity and $\mathcal E$-adequacy taken together do not exclude the self-referentiality of a numbering. While the constraint of strong monotonicity excludes the self-referentiality of numberings, we have shown that it does not block the attainability of m-self-reference in the basic language $\mathcal{L}^0$. As a counterpoint, we have shown that the demand that the numbering is regular does indeed render m-self-reference in $\mathcal{L}^0$ unattainable, but that, on the other hand, some completely decent numberings are non-regular. The obtained results show that the attainability of m-self-reference in $\mathcal{L}^0$ is more sensitive to the underlying formalisation choices than widely believed. Finally, we have shown that this sensitivity also impinges on the formal study of certain principles of self-referential truth. Namely, whether or not certain axiomatic theories of self-referential truth are consistent is not determined by the constraints of $\mathcal E$-adequacy and (strong) monotonicity alone.
	
	In the paper we formulated some open questions that we repeat here.
	\begin{enumerate}[A.]
		\item
		Are there $\pol$-adequate numberings $\xi$ for which $\xi\circ(\underline \cdot)$ is p-time?
		\item Can we find $\pol$-adequate self-referential G\"odel numberings that are monotonic?
		\item Can we find $\pol$-adequate G\"odel numberings which satisfy the Strong Diagonal Lemma~\ref{lem:strongdiaglem} that are strongly monotonic?
		\item Is the numbering $\mathfrak h$ (defined in Remark~\ref{remark:collapsednumbering}) $\mathcal{L}^0$-regular?
	\end{enumerate}
	
	Our paper is a primarily technical contribution providing \emph{data for philosophy}.
	We have only gestured at philosophical motivations of the various constraints. 
	We do believe that the motivations given in the literature for various constraints are extremely thin. 
	However, we submit that this does not undermine the value of our contribution too much,
	since our work will still provide a philosopher who sets out to carefully argue for a
	constraint with an impression of what such a contraint does or does not mean. Moreover,
	both the self-referential numberings themselves and the study of various constraints on
	numberings have some technical interest entirely independent of the philosophical
	motivation of various constraints.
	
	\bibliographystyle{apalike}
	\bibliography{provint,filo,mybib}
	
	\appendix
	
	\section{An Alternative Construction of a Monotonic Self-Referential Numbering for Strings}
	\label{section:construction1}
	We provide an alternative construction of a coding for strings which is self-referential for efficient numerals and monotonic with respect to the sub-expression relation. We believe that the inclusion of this construction is instructive, since in some sense it proceeds in a \enquote{bottom-up} and more direct fashion than the constructions given in Section~\ref{section:construction2}~\&~\ref{section:construction2B}. 
	
	\subsection{The Length-First Ordering and ${\sf gn}_4$}
	\label{subsection:length-first}
	Let $\mathcal A$ be the alphabet of the arithmetical language $\mathcal{L}^0$ introduced in Subsection~\ref{subsection:language} and let $\mathcal A^+$ be $\mathcal A$ extended with a fresh constant $\mathsf{c}$ and a fresh separator-symbol ;. We remind the reader of the fact that
	$\mathcal A$ has 17 letters. Thus, $\mathcal A^+$ has 19 letters.
	We suppose some ordering of $\mathcal A^+$ is given. Let ${\sf gn}_4$ be the length-first ordering of strings of $\mathcal A^+$. We write $|\alpha|$ for the length of $\alpha$ as before.
	
	We remind the reader that, by Lemma~\ref{lem:lengthfirstestimates}, we have
	\[ \frac{19^{|\alpha|}-1}{18} \leq {\sf gn}_4(\alpha)  <  \frac{19^{|\alpha|+1}-1}{18}.\]
	It follows that $|\alpha| \leq {\sf gn}_4(\alpha)$.
	Moreover, whenever $|\alpha| < |\beta|$, then
	${\sf gn}_4(\alpha) < {\sf gn}_4(\beta)$.
	
	We define $\widetilde n$ as before and remind the reader that ${\sf ev}(\widetilde n ) = 2^n -1$ and $|\widetilde n| = 7n+1$ (see Lemma~\ref{lolligesmurf}).
	
	\subsection{On ${\sf gn}_5$}
	We say that $\alpha$ is \emph{acceptable} if it is of the form $\gamma;\delta$, where
	$\gamma$ and $\delta$ are ;-free and where $\gamma$ is a sub-string of $\delta$. 
	(Here we allow $\gamma = \delta$.)
	
	Suppose $\alpha_n=\gamma;\delta$ is acceptable.
	We take $\beta_n := \gamma[\mathsf{c}:= \overline{2^{\,{\sf gn}_4(\delta;\delta)}}]$.
	In other words, when $\delta;\delta = \alpha_k$, then $\beta_n :=  \gamma[\mathsf{c}:= {\sf S}\widetilde k]$. 
	In all other cases we set $\beta_n$ to a \emph{don't care} value, say $\varepsilon$, the empty string.
	
	\begin{lemma}
		$(\beta_n)_{n\in \omega}$ enumerates the $\mathcal A$-strings.
	\end{lemma}
	
	\begin{proof}
		Clearly the $\beta_n$ neither contain ; nor $\mathsf{c}$, so they are $
		\mathcal A$-strings. 
		
		\medent
		Conversely, suppose $\theta$ is an $\mathcal A$-string.
		Let $\alpha_k = \theta;\theta$. Then, clearly, $\beta_k = \theta$.
	\end{proof}
	
	\noindent
	We note that the enumeration $(\beta_n)_{n\in \omega}$ has repetitions.
	We proceed with a technical lemma. 
	
	\begin{lemma}\label{smurferella}
		Suppose $\gamma;\delta$ is adequate and ${\sf gn}_4(\delta;\delta)= n$. Then, 
		$|\widetilde n| > |\delta;\delta| > |\gamma|$, 
		and, hence, ${\sf S}\widetilde{n}$ does not occur in $\gamma$. 
	\end{lemma}
	
	\begin{proof}
		We have:\qedright
		\begin{eqnarray*}
			|\gamma| & < & |\delta; \delta| \\
			& \leq & \frac{19^{|\delta;\delta|}-1}{18}\\
			&\leq & n \\
			& < & 7n+1\\
			& = & |\widetilde{n}|
		\end{eqnarray*}
	\end{proof}
	
	%	\begin{lemma}\label{smurfin}
	%Suppose $\gamma;\delta$ and $\epsilon;\eta$ are acceptable. 
	%Let $\delta;\delta = \alpha_n$ and
	%$\eta;\eta = \alpha_m$. 
	%Suppose also that  $\gamma[c := {\sf S}\widetilde n] = %\epsilon[c:={\sf S} \widetilde m]$.
	%Then, $\widetilde m$ and $c$  do not both occur in $\gamma$.
	%\end{lemma}
	
	%\begin{proof}
	%We assume that  the conditions of the lemma hold. Suppose, 
	%	to obtain a contradiction that
	%	$\widetilde m$ and $c$ both occur in $\gamma$.
	%	It follows that $|\delta;\delta| \geq |\gamma| \geq 7m+2$.
	%	So $n > \frac{19^{7m+2} -1}{18}$. Moreover, we have   
	%$\frac{19^{7|\eta;\eta|+2} -1}{18}< m$ and hence
	%	$|\eta;\eta| < m$. Thus,
	%	\begin{eqnarray*}
	%		|\gamma[c:= {\sf S}\widetilde n] | 
	%& \geq &  |{\sf S}\widetilde % n | \\
	%		& > & 7 \frac{19^{7m+2} -1}{18} +2 \\
	%		& \geq & m \cdot (7m+2) \\
	%		& \geq & |\eta;\eta| \cdot (7m +2) \\
	%		& \geq & |\epsilon| \cdot (7m +2) \\
	%		& \geq & |\epsilon[c:= {\sf S}\widetilde m]|
	%	\end{eqnarray*}
	%	$$	A contradiction with the assumption that  
	%$\gamma[c := {\sf S}\widetilde n] = 
	%\epsilon[c:= {\sf S}\widetilde m]$.
	%\end{proof}
	
	\begin{lemma}\label{gourmetsmurf}
		Consider acceptable $\zeta := \gamma;\delta$ and $\theta := \mu;\eta$. Let ${\sf gn}_4(\delta;\delta)=n$ and
		${\sf gn}_4(\eta;\eta)=m$.
		Suppose that
		$\mathsf{c}$ occurs both in $\gamma$ and in $\mu$ and
		$\nu := \gamma[\mathsf{c} := {\sf S} \widetilde n] = \mu[\mathsf{c}:= {\sf S} \widetilde m]$.
		Then, $\zeta = \theta$.
	\end{lemma}
	
	\begin{proof}
		We assume the conditions of the lemma. 
		Note that $n$ and $m$ cannot be 0.
		In case $n=m$, we see, by Lemmas~\ref{grumpysmurf} and \ref{smurferella}, that both $\gamma$ and $\epsilon$ are the
		result of replacing ${\sf S}\widetilde{n}$ in $\nu$ by $\mathsf{c}$.
		It follows that $\gamma = \mu$. Trivially, $\delta = \eta$.
		So, $\zeta=\theta$.
		
		Suppose $m\neq n$. This contradicts the fact that the largest term of the form ${\sf S}\widetilde{k}$ in $\nu = \gamma[\mathsf{c}:= {\sf S}\widetilde{n}]$ is ${\sf S}\widetilde n$ and that the largest term of the form  ${\sf S}\widetilde{k}$ in $\nu = \mu[\mathsf{c}:= {\sf S}\widetilde{m}]$ is ${\sf S}\widetilde m$.
	\end{proof}
	
	\begin{lemma}\label{kooksmurf}
		Consider any $\mathcal A$-string $\theta \neq \varepsilon$.
		Let $\alpha_m = \theta;\theta$. We have $\beta_m =\theta$. 
		Suppose $\alpha_n = \gamma;\delta$ is acceptable and $\mathsf{c}$ occurs in $\gamma$ and $\beta_n = \theta$. Then $n<m$.
	\end{lemma}
	
	\begin{proof}
		The first claim is trivial. Suppose the conditions of the second claim.  Let
		$\alpha_k = \delta;\delta$.
		We have, by Lemma~\ref{smurferella}, that
		$|\theta;\theta| > |\theta| \geq |{\sf S}\widetilde k| > |\delta;\delta| \geq |\gamma;\delta|$. 
	\end{proof}
	
	\noindent
	Consider any $\mathcal A$-string $\zeta$.
	Let us say that $\zeta$ is \emph{essential} if, for some $n$, we have
	$\zeta=\beta_n$ and $\alpha_n = \gamma;\delta$, where $\mathsf{c}$ occurs in $\gamma$. 
	Otherwise, we call $\zeta$ \emph{inessential}.
	
	The combination of Lemmas~\ref{gourmetsmurf} and \ref{kooksmurf} tells us
	that the enumeration of the $\beta_n$ looks as follows. Suppose $\zeta$ is non-empty.
	If $\zeta$ is inessential it will be enumerated first as $\beta_{{\sf gn}_4(\zeta;\zeta)}$. 
	All occurrences of $\zeta$ in the enumeration will have the form $\beta_{{\sf gn}_4(\zeta;\eta)}$, 
	where $\zeta$ is a sub-string of $\eta$.
	If $\zeta$ is essential, its first occurrence will be $\beta_{{\sf gn}_4(\gamma,\delta)}$,
	where $\gamma,\delta$ is adequate and $\mathsf{c}$ occurs in $\gamma$ and 
	$\zeta = \gamma[\mathsf{c}:= {\sf S}(\widetilde{{\sf gn}_4(\delta;\delta)})]$. 
	After that the enumeration mirrors the inessential case.
	
	\begin{remark}	We note that we could define a many-valued G\"odel numbering ${\sf gn}^\ast(\zeta) :=
		\verz{n\mid \zeta= \beta_n}$. There are no fundamental objections to many-valued G\"odel numberings as long as they are total,
		in the sense that to each string a non-empty set of values is assigned, and injective, in the sense that
		the sets of values assigned to two different strings are disjoint. However, since in our framework, we opted
		for the more conventional choice of functional G\"odel numberings, we will not consider a G\"odel numbering like
		${\sf gn}^\ast$.
	\end{remark}
	
	\noindent	We define, for an $\mathcal A$-string $\theta$: ${\sf gn}_5(\theta)$ is the  smallest $n$ such that $\beta_n= \theta$. 
	
	We could compute ${\sf gn}_5$ as follows.
	\newcommand\textvtt[1]{{\normalfont\fontfamily{cmvtt}\selectfont #1}}
	Let an $\mathcal A$-string $\zeta$ be given.
	\begin{quotation}
		\textvtt{Is $\zeta$ the empty string? If \emph{yes}, ${\sf gn}_5(\zeta) :=0$. If \emph{no}, 
			does $\zeta$ have a sub-string of the form ${\sf S}\widetilde{k}$?
			If \emph{no}, set ${\sf gn}_5(\zeta) := {\sf gn}_4(\zeta;\zeta)$. If \emph{yes}, find the largest such $k$, say it is $n$. 
			Let $\gamma$ be the result of replacing all occurrences of
			${\sf S}\widetilde n$ in $\zeta$ by $\mathsf{c}$. Is $\alpha_n$ of the form $\delta;\delta$, where
			$\gamma;\delta$ is adequate? If \emph{no}, set ${\sf gn}_5(\zeta) := {\sf gn}_4(\zeta;\zeta)$.
			If \emph{yes}, put ${\sf gn}_5(\zeta):= {\sf gn}_4(\gamma;\delta)$.} 
	\end{quotation}
	
	\noindent The G\"odel numbering ${\sf gn}_5$ is not monotonic with respect to the sub-string ordering
	and $\leq$. Suppose, e.g., $n = {\sf gn}_4(\delta;\delta)$.
	Clearly, $\widetilde n$ is inessential, since it does not contain any sub-string of the form ${\sf S}\widetilde k$.
	We remind the reader that $|\widetilde{n}|= 7n+1$. So, we have: 
	\begin{align*}
	{\sf gn}_5(\widetilde n) & = {\sf gn}_4(\widetilde n;\widetilde{n}) \geq \frac{19^{14n+3}-1}{18}>n, \\
	{\sf gn}_5({\sf S}\widetilde n) & = {\sf gn}_4(c;\delta)< {\sf gn}_4(\delta;\delta) = n.
	\end{align*}
	
	\noindent So, ${\sf gn}_5({\sf S}\widetilde n) < {\sf gn}_5(\widetilde{n})$.
	
	We will prove two lemmas about classes of cases where 
	${\sf gn}_5$ is monotonic. 
	
	\begin{lemma}\label{momoasmurf}
		Suppose $\zeta$ is inessential and $\theta$ is a sub-string of
		$\zeta$. Then, ${\sf gn}_5(\theta) \leq {\sf gn}_5(\zeta)$.
	\end{lemma}
	
	\begin{proof}
		Suppose $\zeta$ is inessential. Then, 
		${\sf gn}_5(\theta) \leq {\sf gn}_4(\theta;\theta) \leq {\sf gn}_4(\zeta;\zeta)= {\sf gn}_5(\zeta)$.
	\end{proof}
	
	\noindent
	Let $\mathscr W$ be the set of
	well-formed expressions of the arithmetical language.
	We partition $\mathscr W$  in two classes $\mathscr W_0$ and $\mathscr W_1$.	Here:
	\begin{itemize}
		\item $\mathscr W_0 := \verz{{\sf SS}\zero} \cup \verz{\widetilde n \mid n \in \omega}
		\cup \verz{({\sf SS}\zero \times \widetilde n) \mid n \in \omega}$
		\item $\mathscr W_1 := \mathscr W\setminus \mathscr W_0$
	\end{itemize}  
	
	\begin{lemma}\label{beterwetersmurf}
		Suppose $\theta\in \mathscr W_1$ and $\zeta\in\mathscr W$.
		Suppose further that $\theta$ is a sub-expression of $\zeta$.
		Then, ${\sf gn}_5(\theta) \leq {\sf gn}_5(\zeta)$.
	\end{lemma}
	
	\begin{proof}
		Suppose 
		${\sf gn}_5(\zeta)= {\sf gn}_4(\gamma;\delta)$.
		Let $n := {\sf gn}_4(\delta;\delta)$ and let
		$\theta_0$ be the result of replacing all occurrences of 
		${\sf S} \widetilde n$ in $\theta$ by $\mathsf{c}$. (We include the case here where there are no such occurrences.) 
		
		We claim that $\theta_0$ is a sub-expression of $\gamma$. To see this consider the parse-tree of $\zeta$. 
		Since $\theta$ is
		in $\mathscr W_1$, it cannot occur strictly below a node
		labeled ${\sf S} \widetilde n$. So, relabeling the nodes labeled ${\sf S} \widetilde n$ 
		with $\mathsf{c}$ and removing all nodes below them in the parse-tree of $\zeta$ will result in occurrences of $\theta_0$
		at the places where we originally had occurrences of $\theta$.
		
		It follows that 
		${\sf gn}_5(\theta) \leq {\sf gn}_4(\theta_0;\delta) \leq 
		{\sf gn}_4(\gamma;\delta)= {\sf gn}_5(\zeta)$.
	\end{proof}
	
	\subsection{On ${\sf gn}_6$}
	Finally, we are ready and set to define the desired G\"odel numbering~${\sf gn}_6$.
	\begin{itemize}
		\item
		${\sf gn}_6({\sf SS}\zero) = 2$
		\item
		${\sf gn}_6(\widetilde{n}) := 4n+1$
		\item
		${\sf gn}_6(({\sf SS}\zero \times \widetilde n)):= 4n+3$
		\item
		${\sf gn}_6(\alpha) := 2^{\,{\sf gn}_5(\alpha)}$, 
		if $\alpha$ is in $\mathscr W_1$.
	\end{itemize}
	
	\noindent
	We have:
	\begin{lemma}
		${\sf gn}_6$ is injective on $\mathscr W$.
	\end{lemma}
	
	\begin{proof}
		Suppose $\alpha$ is well-formed. 
		The only case where we could go wrong is where ${\sf gn}_5(\alpha)$ is 0 or 1.
		Since $\alpha$ is well-formed, it is not the empty string.
		So, for some adequate $\gamma;\delta$, where $\gamma$ is well-formed,
		we have ${\sf gn}_5(\alpha) ={\sf gn}_4(\gamma;\delta)$. We note that $|\gamma;\delta|\geq 3$,
		so ${\sf gn}_5(\alpha) ={\sf gn}_4(\gamma;\delta) \geq \frac{19^3-1}{18} > 1$.
	\end{proof}
	
	\noindent
	So ${\sf gn}_6$ is indeed a G\"odel numbering.
	
	\begin{lemma}\label{knutselsmurf}
		Suppose $\alpha$ is in $\mathscr W_0$. Then, ${\sf gn}_6(\alpha) \leq |\alpha|
		\leq {\sf gn}_4(\alpha) \leq {\sf gn}_5(\alpha) $.
	\end{lemma}
	
	\begin{proof}
		Suppose $\alpha$ is in $\mathscr W_0$.
		We note that, for no $k$, we have 
		${\sf S}\widetilde k$ is a sub-string of $\alpha$. So, ${\sf gn}_5(\alpha)
		= {\sf gn}_4(\alpha;\alpha)$.
		It follows that 
		$|\alpha| \leq {\sf gn}_4(\alpha) \leq {\sf gn}_4(\alpha;\alpha)= {\sf gn}_5(\alpha)$. Moreover, we have:\qedright
		\begin{align*}
		{\sf gn}_6({\sf SS}\zero) & = 2 < |{\sf SS}\zero| \\
		{\sf gn}_6(\widetilde{n})& = 4n+1 \leq 7n+1 = |\widetilde{n}|\\
		{\sf gn}_6(({\sf SS}\zero \times \widetilde{n})) & = 4n+3 < 7n+6
		= |({\sf SS}\zero \times \widetilde{n})| 
		\end{align*}
	\end{proof}
	
	\begin{lemma}\label{doktersmurf}
		Suppose ${\sf S}\widetilde n$ is essential. Then,
		${\sf gn}_6(\widetilde n)<{\sf gn}_6({\sf S}\widetilde{n})$
	\end{lemma}
	
	\begin{proof}
		Suppose ${\sf S}\widetilde n$ is essential. Since ${\sf S}\widetilde n$
		cannot have a proper sub-term of the form ${\sf S}\widetilde m$, we find
		that ${\sf gn}_5({\sf S}\widetilde n) = {\sf gn}_4(c;\delta)$ for some $\delta$ that
		contains $\mathsf{c}$. Moreover, $n = {\sf gn}_4(\delta;\delta)$. 
		We have:\qedright
		\begin{eqnarray*}
			{\sf gn}_6(\widetilde n) & = & 4n+1 \\
			& < & 4\frac{19^{2|\delta|+2}-1}{18} + 1 \\
			& < & 2^{\frac{19^{|\delta|+2}-1}{18}} \\
			& \leq  & 2^n 
		\end{eqnarray*}
	\end{proof}
	
	\begin{lemma}
		${\sf gn}_6$ restricted to $\mathscr W$ is monotonic w.r.t.\ the sub-expression ordering and $<$.
	\end{lemma}
	
	\begin{proof}
		Suppose $\theta$ is a  sub-expression of $\zeta$.
		
		\emph{Case 1:}
		Suppose $\theta\in \mathscr W_1$. Then, we are immediately done by Lemma~\ref{beterwetersmurf}.
		
		\emph{Case 2:}
		Suppose $\theta \in \mathscr W_0$.
		
		\emph{Case 2.1:}
		Suppose $\zeta\in \mathscr W_0$. Then, we are done, by inspection of the definition of ${\sf gn}_6$ on $\mathscr W_0$.
		
		\emph{Case 2.2:}
		Suppose $\zeta\in \mathscr W_1$. 
		
		\emph{Case 2.2.1:}
		Suppose $\zeta$ is inessential. Then, by Lemma~\ref{knutselsmurf},
		${\sf gn}_6(\theta) \leq {\sf gn}_5(\theta)$ and, by Lemma~\ref{momoasmurf},
		${\sf gn}_5(\theta) \leq {\sf gn}_5(\zeta)$. Hence,
		\[ {\sf gn}_6(\theta) \leq {\sf gn}_5(\theta) \leq {\sf gn}_5(\zeta) < 2^{{\sf gn}_5(\zeta)}= {\sf gn}_6(\zeta).\]
		
		\emph{Case 2.2.2:}
		Suppose $\zeta$ is essential. It follows that 
		${\sf gn}_5(\zeta)= {\sf gn}_4(\gamma;\delta)$, where
		$\mathsf{c}$ occurs in $\gamma$. Let $n := {\sf gn}_4(\delta;\delta)$.
		By Lemma~\ref{smurferella}, we see that $|\widetilde n| > |\gamma|$.
		It follows that any sub-term of $\gamma$ that is in $\mathscr W_0$ must
		be a sub-term of $\widetilde n$. Hence, any sub-term of $\zeta$ that is in
		$\mathscr W_0$ must be in $\widetilde n$. So, by Case 2.1, ${\sf gn}_6(\theta) \leq {\sf gn}_6(\widetilde n)$.
		Also, clearly, ${\sf S} \widetilde n$ is essential, so, by Lemma~\ref{doktersmurf},
		${\sf gn}_6(\widetilde n) < {\sf gn}_6({\sf S}\widetilde{n})$.
		Finally, by Case 1: ${\sf gn}_6({\sf S}\widetilde{n})\leq {\sf gn}_6(\zeta)$.
		So, ${\sf gn}_6(\theta) \leq {\sf gn}_6(\zeta)$.
	\end{proof}
	
	We end with the obvious insight that ${\sf gn}_6$ is self-referential.
	\begin{lemma}
		Consider any arithmetical formula $A(x)$ in which there is a free occurrence of $x$. Let $n = {\sf gn}_4(A(\mathsf{c});A(\mathsf{c}))$
		and let $k = 2^n$. Then,
		$k = {\sf gn}_6(A(\overline k))$.
	\end{lemma}
	
	\noindent
	We may conclude:
	
	\begin{theorem}
		\label{theorem:gn2selfref}
		The function ${\sf gn}_6$ is a monotonic G\"odel numbering that is self-referential for dyadic numerals. 
	\end{theorem}
	
	\noindent
	We note that our G\"odel numbering is elementary and the tracking functions for the connectives also will be elementary.
\end{document}